\documentclass[12pt,a4paper,twoside]{article}

\usepackage[english]{babel}
\usepackage{amsmath}
\usepackage{amssymb}
\usepackage{amsfonts}
\usepackage{amsthm}

\theoremstyle{plain}
\newtheorem{theorem}{Theorem}
\newtheorem{proposition}[theorem]{Proposition}
\newtheorem{corollary}[theorem]{Corollary}
\newtheorem{lemma}[theorem]{Lemma}
\newtheorem{property}[theorem]{Property}

\theoremstyle{definition}
\newtheorem{remark}[theorem]{Remark}

\numberwithin{equation}{section}
\numberwithin{theorem}{section}

\DeclareMathOperator{\ap}{ap}
\DeclareMathOperator{\dom}{dom}

\DeclareMathOperator{\Bx}{Box}
\DeclareMathOperator{\spn}{span}

\begin{document}

\title{APPROXIMATE DIFFERENTIABILITY OF MAPPINGS OF CARNOT--CARATH\'EODORY SPACES
\thanks{The research was partially supported by the Russian Foundation
for Basic Research (Grants 10--01--00662-a and 11--01--00819-a),  the State Maintenance Program
for Young Russian Scientists and the Leading Scientific Schools of the
Russian Federation (Grant NSh-921.2012.1), and  the Federal Target Grant "Scientific and educational
personnel of innovation Russia" for 2009-2013 (Agreements
No. 8206 and 8212).}}
\date{\empty}
\author{S.~G. Basalaev, S.~K. Vodopyanov
}

\maketitle



\noindent\textbf{Abstract:} {
  We study the approximate differentiability of measurable mappings of
  Carnot--Carath\'eodory spaces. We show that the approximate
  differentiability almost everywhere is equivalent to
  the approximate differentiability along the basic horizontal vector
  fields almost everywhere. As a geometric tool we prove the
  generalization of Rashevsky--Chow theorem for $C^1$-smooth
  vector fields. The main result of the paper extends theorems
  on approximate differentiability proved by Stepanoff (1923, 1925)
  and Whitney (1951) in Euclidean spaces and by Vodopyanov (2000)
  on Carnot groups.
}

\tableofcontents

\section*{Introduction}
\addcontentsline{toc}{section}{Introduction}

In 1919  Rademacher proved a theorem that is the well-known
result of the theory of functions of real variable.
\begin{theorem}[\cite{Bib:Rademacher}] \label{Theorem:Rademacher}
    If $U$ is an open subset in
    $\mathbb{R}^n$ and $f : U \to \mathbb{R}^m$ is a Lipschitz mapping
    then $f$ is differentiable at almost all points of the set $U$.
\end{theorem}

The result permits many enhancements and generalizations. The most
natural is to have an arbitrary measurable set as the domain of
the function together with a weaker assumption on the function.
Such a result is the Stepanoff theorem:
\begin{theorem}[\cite{Bib:StepDiff}] \label{Theorem:Stepanoff}
    If $A \subset \mathbb{R}^n$ is a measurable set and the function
    $f : U  \to \mathbb{R}^m$ satisfies the condition
    \begin{equation} \label{Eq:StepanoffCond}
        \varlimsup_{x \to a} \frac {|f(x)-f(a)|} {|x-a|} < \infty
        \quad \text{at every point} \ a \in A,
    \end{equation}
     then $f$ is differentiable at almost all points of the set $A$.
\end{theorem}

{\it The density} of a measurable set $Y \subset \mathbb{R}^n$
at a point $x \in \mathbb{R}^n$ is a limit
$$
    \lim_{r \to +0} \frac{\mathcal{H}^n(Y \cap B(x,r))}{\mathcal{H}^n(B(x,r))},
$$
in case it exists (here $\mathcal{H}^n$ is $n$-dimensional Hausdorff measure).

It is known that almost all points of a measurable set $Y$ are
the density points (i.~e. the density of the set is $1$ at those points)
and almost all points of the set $\mathbb{R}^n \setminus Y$ are the points
of the density $0$.

A value $y \in \mathbb{R}^m$ is called {\it the approximate limit}
of a function $f : E \subset \mathbb{R}^n \to \mathbb{R}^m$
at a density point $x_0 \in E$
$\bigl($denoted by $y = \ap \lim\limits_{x \to x_0} f(x)\bigr)$ if
the set $E \setminus f^{-1}(W)$ have the density $0$ at the point
$x_0$ for every neighborhood $W \subset \mathbb{R}^m$ of the point~$y$.
The approximate limit is unique \cite{Bib:Federer}.

The idea of the approximate limit is tightly connected with the
fundamental notion of the geometric measure theory: the notion
of measurability. Precisely, for a mapping of the Euclidean spaces
to be measurable, it is necessary and sufficient to be approximately
continuous almost everywhere (see, for instance, \cite{Bib:Federer}).

If we consider the convergence of the relation $\dfrac{f(x+tv)-f(x)}{t}$
to the value $L(v)$ of a linear mapping $L : \mathbb{R}^n \to \mathbb{R}^m$ in
different topologies of the unit ball $B(0,1) \subset \mathbb{R}^n$ then
we proceed to different notions of differentiability. The convergence to $L$
in the uniform topology $C(B(0,1))$ gives us the classical differentiability.
The convergence to $L$ by measure gives just the notion of approximate
differentiability of the Euclidean spaces, see for instance \cite{Bib:Reshetnyak}.

With the approximate differential introduced by Stepanoff, the following
result was obtained in his work:
\begin{theorem}[\cite{Bib:StepAppr}]
	The function is approximately differentiable almost everywhere if and
	only if it has approximate derivatives with respect to each variable
	almost everywhere.
\end{theorem}

It worth noting that if a mapping has a classical differential
then it has an approximate one and these differentials coincide.
Therefore, the approximate differential generalizes the concept of
the classical differentiability.

With use of the approximate differential Theorem~\ref{Theorem:Stepanoff}
can be further extended in the following direction.
For doing this we apply a result of \cite{Bib:Federer}:
\begin{theorem}
    If $A \subset \mathbb{R}^n$, $f : A \to \mathbb{R}^m$ and
    \begin{equation} \label{Eq:ApprCond}
        \ap \varlimsup_{x \to a} \frac {|f(x) - f(a)|} {|x-a|} < \infty
        \quad \text{for every point}\ a \in A,
    \end{equation}
    then $A$ is a union of the disjoint sequence of the measurable sets
    $A_i$ and a set of measure zero such that every restriction
    $f|_{A_i}$ is a Lipschitz mapping.
\end{theorem}

Hence, for a function $f$ meeting the condition \eqref{Eq:ApprCond},
by Theorem~\ref{Theorem:Rademacher}, we have every restriction $f|_{A_i}$
being differentiable almost everywhere in $A_i$. The density points
for the set $A_i$ also are the density points for the set $A$.
Therefore, one can conclude that the mapping $f$ is approximately
differentiable almost everywhere in $A$.

The condition \eqref{Eq:ApprCond} is the weakest because it obviously
holds for the approximately differentiable function.

The final representation of the theorem is how it was stated by Whitney
\begin{theorem}[\cite{Bib:Whitney}] \label{Theorem:Whitney}
    Let the set $P \subset \mathbb{R}^n$ be measurable and bounded,
    $f : P \to \mathbb{R}^m$ be a measurable function.
    The following conditions are equivalent:

    $1)$ the mapping $f$ is approximately differentiable almost everywhere in $P$;

    $2)$ the mapping $f$ has approximate derivatives with respect to
        each variable almost everywhere in $P$;

    $3)$ there is a countable family of the disjoint sets
        $Q_1,Q_2,\dots$ such that
        $
            \vert P \setminus \bigcup\limits_{i=1}^{\infty} Q_i \vert = 0
        $
        and every restriction $f|_{Q_i}$ is a Lipschitz mapping;

    $4)$ for every $\varepsilon > 0$, there are a closed set $Q \subset P$
        such that $|P \setminus Q| < \varepsilon$ and a $C^1$-smooth mapping
        $g : P \to \mathbb{R}^m$ such that $g = f$ in $Q$.
\end{theorem}

An appropriate concept of differentiability for mappings of Carnot
groups was introduced by P. Pansu in \cite{Bib:Pansu}. Now it is
called the $\mathcal{P}$-differentiability. It was introduced for
some results of the theory of quasiconformal mappings to establish
\cite{Bib:Pansu, Bib:KR}. Some classes of
$\mathcal{P}$-differentiable mappings of Carnot groups were described
in \cite{Bib:Vod-Ukhlov, Bib:Vod-Groups, Bib:Magnani} with a purpose
to obtain some formulas of geometric measure theory and some crucial
results of quasiconformal analysis
\cite{Bib:Vod96, Bib:Vod-Ukhlov2, Bib:Vod99, Bib:Vod07-1,
Bib:Vod07-2, Bib:Pauls}.

Later, in \cite{Bib:Vod-Spaces, Bib:Karm-Vod} concept of
$\mathcal{P}$-differentiability was extended for mappings of
Carnot--Carath\'eodory spaces for proving Rademacher and Stepanoff
type theorems.

In this work we obtain a partial generalization of Theorem~\ref{Theorem:Whitney}
to mappings of Carnot--Carath\'eodory spaces.
\begin{theorem} \label{Theorem:WorkResult}
	Let $\mathcal{M}$, $\widetilde{\mathcal{M}}$
	be Carnot--Carath\'eodory spaces,
	$E \subset \mathcal{M}$ be a measurable subset
	of $\mathcal{M}$ and
    $f : E \to \widetilde{\mathcal{M}}$ be a measurable mapping.
    The following conditions are equivalent:

    $1)$ the mapping $f$ is approximately differentiable
        almost everywhere in $E$;

    $2)$ the mapping $f$ has approximate derivatives along the basic
        horizontal vector fields almost everywhere in $E$;

    $3)$ there is a sequence of the disjoint sets
        $Q_1, Q_2, \dots$ such that
        $\bigl|E \setminus \bigcup\limits_{i=1}^{\infty} Q_i\bigr| = 0$ and
        every restriction $f|_{Q_i}$ is a Lipschitz mapping.
\end{theorem}

A proof of Theorem~\ref{Theorem:WorkResult} is a significant
modification of the arguments of the work \cite{Bib:Vod-Groups}
where the similar result was proved for mappings of Carnot groups.
In the proof we essentially use metric properties of the initial
and nilpotentized vector fields discovered in
\cite{Bib:Karm-Vod, Bib:Karm, Bib:KarmG, Bib:Greshnov}.

\section{Geometry of Carnot--Carath\'eodory spaces}

We split our work in four sections.
In the first one we give the basic notions and structures concerning
Carnot--Carath\'eodory spaces. In Subsections~\ref{Section:Coord1}
and \ref{Section:Coord2} we have a look at different ways of
specifying a metric and coordinate system in the Carnot--Carath\'eodory
spaces. In Subsection~\ref{Section:PhiCoord} we build a special coordinate
system of the second kind based on the compositions of the integral
lines of the horizontal vector fields. As the consequence of this result
we obtain Chow--Rashevsky theorem for $C^1$-smooth vector fields. We
formulate also local approximation theorem for Carnot--Carath\'eodory
metric.

In Section~\ref{Section2} we introduce definitions of measure,
approximate limit, differentiability and approximate
differentiability and formulate necessary results obtained earlier.

The third section is devoted to the proof of the theorem on approximate
differentiability. We state the theorem and show trivial implications.
Then we formulate the key step of the theorem. Main steps of its
proof are carried out in separate lemmas. In this proof we make use of
special coordinate system of the 2nd kind
$(a_1,\dots,a_N) \mapsto \Phi_N(a_N)\circ\dots\circ\Phi_1(a_1)$
constructed in Subsection~\ref{Section:PhiCoord}. First,
in Subsection~\ref{Section:ApproxDerivatives} we show that function having
approximate derivatives along the basic horizontal vector fields 
has approximate derivatives along the vector fields $Y_k(t)$
which generate the coordinate functions $\Phi_k(t) = \exp(Y_k(t))$.
In the next subsection with use of this coordinate system we build a
mapping of local Carnot groups and study its properties. Finally,
in Subsection~\ref{Section:MainTheoremProof} we prove that
this mapping is really the differential of the initial mapping.

As an application of our results, in the last section we prove an area
formula for approximately differentiable mappings.

\subsection{Carnot--Carath\'eodory spaces}
\label{Section:CarnotManifold}

Recall the definition of
Carnot--Carath\'eodory space satisfying the condition of
the equiregularity (\cite{Bib:Gromov, Bib:NSW, Bib:Karm-Vod}).
Fix a connected Riemannian $C^\infty$-manifold $\mathcal{M}$ of topological
dimension $N$. The manifold $\mathcal{M}$ is called a
{\it Carnot--Carath\'eodory space} if the tangent bundle $T \mathcal{M}$
has a filtration
$$
	H \mathcal{M} = H_1 \mathcal{M} \subsetneq \dots \subsetneq
	H_i \mathcal{M} \subsetneq \dots \subsetneq
	H_M \mathcal{M}	= T \mathcal{M}
$$
by subbundles such that every point $g \in \mathcal{M}$ has a neighborhood
$U(g) \subset \mathcal{M}$ equipped with a collection of $C^{1}$-smooth
vector fields $X_1, \dots, X_N$, constituting a basis of $T_v \mathcal{M}$
in every point 	$v \in U(g)$ and meeting the following two properties.
For every $v \in U(g)$,

$(1)$ $H_i \mathcal{M} (v) = H_i (v) = \spn \{ X_1(v), \dots, X_{\dim H_i}(v)\}$
is a subspace of $T_v \mathcal{M}$ of a constant dimension $\dim H_i$,
$i = 1, \dots, M$;

$(2)$ $H_{j+1} = \spn \{ H_j, [H_1, H_{j}], [H_2, H_{j-1}], \dots, [H_{k}, H_{j+1-k}] \}$
where $k = \lfloor \frac {j+1}{2} \rfloor$, $j=1,\dots,M-1$.

The subbundle $H \mathcal{M}$ is called {\it horizontal}.

The number $M$ is called the {\it depth} of the manifold $\mathcal{M}$.

The {\it degree} $\deg X_k$ is defined as $\min \{ m \mid X_k \in H_m\}$.

\begin{remark}
The condition $(2)$ implies that we have the following ``commutator table'':
\begin{equation} \label{Eq:CommCoeff}
	[X_i, X_j](v)
	= \sum_{k:\:\deg X_k \leq \deg X_i + \deg X_j} c_{ijk}(v) X_k(v).
\end{equation}
Note, that \eqref{Eq:CommCoeff} is weaker than condition $(2)$ as it just
implies $[H_i, H_j]\subseteq H_{i+j}$.
\end{remark}

\subsection{The coordinates of the 1st kind}
\label{Section:Coord1}

In the sequel we denote by
$B_{e}(a,r)$ an open Euclidean ball centered at the point
$a\in\mathbb R^N$ and with a radius $r$. From the theorems
on smooth dependence of solutions of ordinary differential equations
on a parameter it follows~(see e.~g.~\cite{Bib:Arnold}) that the mapping
$$
    \theta_g: (x_1,\dots,x_N)\to
        \exp\Bigl(\sum\limits_{i=1}^{N}x_iX_i\Bigr)(g), \quad \theta_g(0)=
            \theta_g(0,\dots,0)=g,
$$
is a $C^1$-smooth
diffeomorphism of a ball
$B_e(0,\varepsilon_g)$ in $\mathbb{R}^N$, where $\varepsilon_g$ is a positive number
small enough, into the neighborhood $O_g$ of the point $g \in \mathcal{M}$.

The collection of numbers $\{x_i\}$, $i=1,\dots, N$, where
$(x_1,\dots,x_N)=\theta_g^{-1}u\in B_e(0,\varepsilon_g)$, is called
{\it the coordinates of the $1$st kind} of the point
$u=\exp\Bigl(\sum\limits_{i=1}^{N}x_iX_i\Bigr)(g)$.

The neighborhood $U(g_0)$ of the point $g_0$ can be chosen so that
$U(g_0) \subset \bigcap\limits_{g \in U(g_0)} O_g$. Then for every
couple of points $u, g \in U(g_0)$ there is the unique tuple of
numbers $(y_1,\dots,y_N)$ such that
$u=\exp\Bigl(\sum\limits_{i=1}^{N}y_i X_i\Bigr)(g)$. For
every couple of points $u$ and $g$ define the non-negative quantity
$$
    d_\infty(u, g) = \max \bigl\{|y_i|^{1/{\deg X_i}} : i = 1,\dots ,N \bigr\}.
$$

An open ball in quasidistance $d_\infty$ of radius $r$ with center in
$g \in \mathcal{M}$ we denote as $\Bx(g,r)$.

\subsection{Local geometry of Carnot--Carath\'eodory spaces}
\label{Section:TangentCone}

Using the normal coordinates $\theta_g^{-1}$, define the
{\it dilation} $\Delta^g_\varepsilon : B(g,r) \to B(g, \varepsilon r)$,
$0 < r \leq r_g$: to an element
$x = \exp \Bigl( \sum\limits_{i=1}^N x_i X_i \Bigr)(g)$ we assign
$$
	\Delta^g_\varepsilon x = \exp \Bigl(
	\sum_{i=1}^N x_i \varepsilon^{\deg X_i} X_i \Bigr) (g)
$$
in the case when the right-hand size makes sense. The following theorem
generalizes a result established under additional smoothness of vector
fields in \cite{Bib:Metivier, Bib:RS, Bib:Gromov}.

\begin{theorem}
\label{Prop:NilpLieAlgebra}
Let $g$ be  a point  in the Carnot--Carath\'eodory space $\mathcal{M}$.
The following statements hold$:$

$(1)$ Coefficients
$$
    \widehat c_{ijk} =
    \begin{cases}
        c_{ijk}(g), & \text{if}\ \deg X_i + \deg X_j = \deg X_k; \\
        0 & \text{otherwise};
    \end{cases}
$$
where $c_{ijk}(\cdot)$ are the functions from the commutator table
\eqref{Eq:CommCoeff}, define the structure of nilpotent graded Lie
algebra on $T_g \mathcal{M}$.

$(2)$ There are vector fields $\{ \widehat X^g_i \}$
with the initial conditions
$\widehat X^g_i(g) = X_i(g)$, $i=1,\dots,N$, taking place in
$\Bx(g,r_g)$ that constitute a basis of the nilpotent graded
Lie algebra $V(g)$ with the following ``commutator table''$:$
\begin{equation} \label{Eq:NilpCommTable}
    [\widehat X^g_i, \widehat X^g_j] =
    \sum_{k=1}^{N} \widehat c_{ijk} \widehat X^g_k =
    \sum_{\deg X_k = \deg X_i + \deg X_j} c_{ijk}(g) \widehat X^g_k.
\end{equation}

$(3)$  For $x \in \Bx(g, r_g)$  consider the vector fields
$$
X^\varepsilon_i(x) = (\Delta^g_{\varepsilon^{-1}})_* \varepsilon^{\deg X_i}
	X_i(\Delta^g_\varepsilon x), \quad i=1,\dots,N.
$$
Then the following equality holds
\begin{equation} \label{Eq:EpsFieldExpansion}
  X^\varepsilon_i(x) = \widehat X^g_i(x) + \sum_{j=1}^N a_{ij}(x) \widehat X^g_j(x)
\end{equation}
where $a_{ij}(x) = o(\varepsilon^{\max\{0, \deg X_j - \deg X_i\}})$
in $x \in \Bx(g, r_g)$ as $\varepsilon\to 0$.

Moreover, given a compact set $\mathcal K\subset \mathcal M$ there exists
$r>0$ such that the relation \eqref{Eq:EpsFieldExpansion} holds for all
$g\in\mathcal K$ with $x\in \Bx(g, r)$
and $o(\cdot)$ is uniform  in~$g$ belonging to  $\mathcal{K}$ as $\varepsilon\to0$.
\end{theorem}

The first statement of theorem is proved in \cite{Bib:Karm-Vod}. The second
follows from the classical Lie theorem \cite{Bib:Lie, Bib:Postnikov}. The third
statement is obtained in \cite{Bib:KarmG} for $C^{1,\alpha}$-smooth vector
fields and in \cite{Bib:Greshnov} for $C^1$-smooth vector fields.

The equality \eqref{Eq:EpsFieldExpansion} implies Gromov's nilpotentization 
theorem with respect to the coordinates of the first kind.
Notice that for the first time it was formulated 
in \cite[p.~130]{Bib:Gromov} in the coordinates of the second kind.

\begin{theorem}[\cite{Bib:KarmG,Bib:Greshnov}]
The uniform convergence $X^\varepsilon_i \to \widehat X^g_i$
as $\varepsilon \to 0$, $i=1,\dots,N$, holds at the points of
$\Bx(g,r_g)$ and this convergence is uniform in $g$ belonging
to some compact neighborhood.
\end{theorem}

The Lie algebra from Theorem~\ref{Prop:NilpLieAlgebra} can be constructed
as a graded nilpotent Lie algebra $V'$ of vector fields
$(\widehat X^g_j)'$ in $\mathbb{R}^N$, $j=1,\dots,N$,
such that the exponential mapping
$	
	(x_1, \dots, x_N) \mapsto 
		\exp \Bigl(\sum\limits_{i=1}^N x_i (\widehat X^g_j)' \Bigr) (0)
$
equals identity \cite{Bib:Postnikov, Bib:Lanconelli}.

The connected simply connected Lie group $\mathbb G_g \mathcal{M}$ with the
nilpotent graded Lie algebra $V'$ is called the {\it nilpotent tangent cone}
of the Carnot--Carath\'eodory space $\mathcal{M}$ at the point $g \in \mathcal{M}$.
The condition $(2)$ from the definition of Carnot--Carath\'eodory space provides
that $\mathbb G_g \mathcal{M}$ is a {\it Carnot group}, i.~e. if
we denote $V_k = \spn \{ \widehat X^g_i : \deg X_i = k \}$ then
\begin{align*}
	V' = V_1 \oplus V_2 \oplus \dots \oplus V_M, \quad
	& [V_1, V_k] = V_{k+1}, \quad k=1,\dots,M-1, \\
	& [V_1, V_M] = \{0\}.
\end{align*}

By means of the exponential map we can push-forward the vector
fields $(\widehat X^g_j)'$ onto some neighborhood of $g \in \mathcal{M}$
for obtaining the vector fields
$\widehat X^g_j ( \theta_g(x)) = D \theta_g (x) \langle (\widehat X^g_j)' \rangle$.

To the Carnot group $\mathbb G_g \mathcal{M}$ there corresponds a
{\it local Carnot group} $\mathcal G^g$ with the nilpotent Lie algebra
with the basic vector fields $\widehat X^g_1,\dots,\widehat X^g_N$.
Define it so that the mapping $\theta_g$
is a {\it local group isomorphism} between some neighborhoods of
the identity elements
of the groups $\mathbb G_g\mathcal M$ and ${\mathcal G}^g$.
The group operation for the elements
$x = \exp \Bigl(\sum\limits_{i=1}^N x_i \widehat X^g_i \Bigr)(g) \in \mathcal{G}^g$
and
$y = \exp \Bigl(\sum\limits_{i=1}^N y_i \widehat X^g_i \Bigr)(g) \in \mathcal{G}^g$
is defined by means of local group isomorphism:
\begin{align*}
  x \cdot y
  & = \exp \Bigl(\sum_{i=1}^N y_i \widehat X^g_i \Bigr) \circ
    \exp \Bigl(\sum_{i=1}^N x_i \widehat X^g_i \Bigr)(g) \\
  & = \theta_g \circ
    \exp \Bigl(\sum_{i=1}^N y_i (\widehat X^g)'_i \Bigr) \circ
    \exp \Bigl(\sum_{i=1}^N x_i (\widehat X^g)'_i \Bigr)(0).
\end{align*}

Define the one-parameter dilation group $\delta^g_t$ on $\mathcal{G}^g$:
\par\noindent to the element
$x = \exp \Bigl(\sum\limits_{i=1}^N x_i \widehat X^g_i \Bigr)(g) \in \mathcal{G}^g$,
there corresponds
$$
    \delta^g_t x = \exp \Bigl(
        \sum_{i=1}^{N} x_i t^{\deg X_i}\widehat X^g_i
    \Bigr)(g) \in \mathcal{G}^g, \quad t \in (0, t(x)).
$$
The relation $\delta^g_t x \cdot \delta^g_\tau x = \delta^g_{t\tau} x$
is defined for $t$, $\tau$ such that $t, \tau, t\tau \in (0, t(x))$.

We extend the definition of $\delta^g_t$ on the negative $t$,
setting $\delta^g_t x = \delta^g_{|t|}(x^{-1})$ for $t < 0$.

Since the local Carnot group $\mathcal{G}^g$ itself is a Carnot--Carath\'eodory
space with the collection of vector fields $\{ \widehat X^g_j \}$, it
is endowed with the quasidistance $d_\infty^g(x,y)$.

Throughout the paper we use the following properties.
\begin{property}[\cite{Bib:Karm-Vod}]
Geometric properties of the local Carnot group:

$(1)$ The mapping $\delta^g_t$ is a group automorphism:
for all elements $x, y \in \mathcal{G}^g$ and numbers
$t \in (0, \min\{ t(x), t(y), t(x \cdot y) \})$
we have $\delta^g_t x \cdot \delta^g_t y = \delta^g_t (x \cdot y)$.

$(2)$ The function $\mathcal{G}^g \owns x \to d_\infty^g (g,x)$ is a local
homogeneous norm on $\mathcal{G}^g$, i.~e., it meets the following conditions:

$(a)$ $d_\infty^g(g,x) \geq 0$ for $x \in \mathcal{G}^g$ and $d_\infty^g(g,x)=0$
if and only if $x=g$;

$(b)$ $d_\infty^g (g, \delta^g_tx) = t d_\infty^g(g,x)$ for every
$t \in (0, t(x))$;

$(c)$ $d_\infty^g(g,x \cdot y) \leq Q_1 \bigl( d_\infty^g(g,x) + d_\infty^g(g,y) \bigr)$
for all $x$, $y$, $x \cdot y \in \mathcal{G}^g$. The constant $Q_1$ is bounded
with respect to $g$ in some compact set in $\mathcal{M}$.

$(3)$ The quantity $d_\infty^g (a,b) = d_\infty^g (g, b^{-1} \cdot a)$
is a left invariant distance on
$\mathcal{G}^g$$:$ $d_\infty^g(x \cdot a,x \cdot b) = d_\infty^g(a,b)$ for all
$a$, $b$, $x \in \mathcal{G}^g$ for which the left- and right-hand
sides of the equality make sense.
\end{property}

\begin{property}[\cite{Bib:Karm-Vod}]
Let $g \in \mathcal{M}$. Then
$$
	\exp \Bigl( \sum_{i=1}^N a_i X_i \Bigr)(g)
	= \exp \Bigl( \sum_{i=1}^N a_i \widehat X^g_i \Bigr)(g)
$$
for all $|a_i| < r_g$, $i=1,\ldots,N$.
\end{property}

Observe, that the latter implies $d^g_\infty(g,x) = d_\infty(g,x)$.

\begin{proposition}[\cite{Bib:Karm-Vod, Bib:Karm-Vod1}] \label{propd1}
	The quantity $d_\infty$ is a quasimetric in the sense
	of {\rm \cite{Bib:NSW}} that is the following
	relations hold for all points of the neighborhood $U(g_0)$$:$

	\noindent{$1)$} $d_\infty(u,g)\geq0$,
	$d_\infty(u,g)=0$ if and only if $u=g$$;$

	\noindent{$2)$} $d_\infty(u,g)= d_\infty(g,u)$$;$

	\noindent{$3)$} there is a constant $Q\geq 1$ such that,
	for every triple of points $u$, $w$, $v\in U(g_0)$, we have
	$$
		d_\infty(u,v) \leq Q( d_\infty(u,w) + d_\infty(w,v) ).
	$$
\end{proposition}

An essential difference between the geometry of a sub-Riemannian space
and the geometry of a Riemannian space is that the metrics of the initial
space and of the nilpotent tangent cone are not bi-Lipschitz equivalent.
Therefore, in studying the questions of the local behavior of the geometric
objects, it is important to know estimates of the deviation of one metric
from another.

\begin{theorem}[{\cite[Theorem 8]{Bib:Karm-Vod1}}]
\label{Th:PathsEstimate}
Assume that $g$, $w_0 \in U(g_0)$ satisfy $d_\infty(g, w_0) = C \varepsilon$.
For a fixed $L \in \mathbb{N}$, consider the points
$$
  \widehat w^\varepsilon_j
   = \exp \Bigl( \sum_{i=1}^N w_{i,j} \varepsilon^{\deg X_j} \widehat X^g_j \Bigr)
   ( \widehat w^\varepsilon_{j-1} ),	
  \quad w^\varepsilon_j
   = \exp \Bigl( \sum_{i=1}^N w_{i,j} \varepsilon^{\deg X_j} X_j \Bigr)
   ( w^\varepsilon_{j-1} ),
$$
$\widehat w^\varepsilon_0 = w^\varepsilon_0 = w_0$, $j = 1, \dots, L$. Then
$$
	\max \{ d^g_\infty ( \widehat w^\varepsilon_L, w^\varepsilon_L),
	d_\infty(\widehat w^\varepsilon_L, w^\varepsilon_L) \}
	= o(\varepsilon) \quad \text{~as~} \varepsilon \to 0,
$$
where $o(\varepsilon)$ is uniform in $g$, $w_0 \in U(g_0)$ and
$\{ w_{i,j} \}$, $i=1,\dots,N$, $j=1,\dots,L$, in some compact
neighborhood of $0$ and $\varepsilon > 0$.
\end{theorem}

\begin{theorem}[{\cite[Theorem 6]{Bib:Karm-Vod1}}]
\label{Theorem:MetricsDeviation}
Consider points $g \in \mathcal{M}$ and $u$, $v \in \Bx(g,\varepsilon)$,
where $\varepsilon \in (0, r_g)$. Then
$$
  |d^g_\infty(u,v) - d_\infty(u,v)| = o(\varepsilon)
  \quad \text{~as~} \varepsilon \to 0,
$$
where $o(\varepsilon)$ is uniform in $u$, $v \in \Bx(g,\varepsilon)$
and $g$ belonging to some compact set.
\end{theorem}

\subsection{The coordinates of the 2nd kind}
\label{Section:Coord2}
In the neighborhood of a point $g_0$ consider the same family
of the basic vector fields
$\{X_1, \dots, X_{\dim H_1}, X_{\dim H_1+1}, \dots, X_N\}$
as in definition of the coordinates of the first kind. It is known
that the mapping
\begin{equation}\label{2kind}
	(a_1,\dots,a_N) \mapsto
	\exp(a_N X_N)\circ\dots\circ\exp(a_1 X_1)(g)
\end{equation}
is a $C^1$-diffeomorphism of some neighborhood
$B_e(0,\varepsilon)\subset\mathbb R^N$ to a neighborhood
$V(g)$ of $g$ (so called \textit{coordinates of the second kind}).
Similarly to the case of the coordinates of the first kind we can choose
a neighborhood $U(g_0)$ such that 
$U(g_0)\subset\bigcap\limits_{g\in U(g_0)}V(g)$.

For the points $u,g\in U(g_0)$, $u=\exp(a_N X_N)\circ \dots \circ\exp(a_1 X_1)(g)$,
by means of the coordinates of the 2nd kind we can define a quantity
$$
    d_2(u,g)=\max\bigl\{|a_i|^{1/\deg X_i} : i=1,\dots,N\bigr\}.
$$
Next we show that the quantity $d_2(u,g)$ is comparable with the quasimetric
$d_\infty(u,g)$ in a neighborhood $U(g_0)$ i.~e.
\begin{equation}\label{comp}
    c_1 d_\infty(u,g)\leq d_2(u,g)\leq c_2 d_\infty(u,g)
\end{equation}
for all points $u$, $g\in U(g_0)$ and positive constants $c_1$ and $c_2$
independent of $u$, $g\in U(g_0)$.

\begin{remark} \label{Rem:MetricsEquality}
 For Carnot groups the equivalence of $d_\infty$ and $d_2$
  is known (see, for instance, \cite{Bib:Folland-Stein}).
  This means that if $d^g_\infty$ and $d^g_2$ are quasimetrics in the local
  Carnot group $\mathcal{G}^g$, $g \in \mathcal{M}$, then there
  are constants $c^g_1$ and $c^g_2$ such that
  \begin{equation} \label{Eq:LocalMetricsEquality}
    c^g_1 d^g_\infty(u,v) \leq d^g_2(u,v) \leq c^g_2 d^g_\infty(u,v)
  \end{equation}
  for all $u$, $v \in \mathcal{G}^g$.
\end{remark}

\begin{proposition} \label{Prop:MetricsEquality}
    There are constants $c_1$ and $c_2$ such that
    inequalities ~$\eqref{comp}$ hold for all points $u$, $g$
    in some neighborhood $U(g_0)$ in which quasimetrics
    $d_\infty$ and $d_2$ are defined.
\end{proposition}

\begin{proof}
Let $u$, $g \in U(g_0)$ be arbitrary points and $d_2(u,g)=r$.
Assuming that $y_0 = g$,
$y_1 = \exp (a_1 X_1)(y_0), \dots, y_N = \exp(a_N X_N)(y_{N-1})$
from the generalized triangle inequality (see Proposition \ref{propd1})
we have the following relations
\begin{align}
	d_\infty(u,g)
	& \leq Q^{N-1} \Bigl( \sum_{i=1}^N d_\infty(y_k,y_{k-1}) \Bigr) \notag \\
	& \leq Q^{N-1} \Bigl( \sum_{i=1}^N |a_i|^{\frac 1 {\deg X_i}} \Bigr)
	\leq N Q^{N-1} r=N Q^{N-1} d_2(u,g). \label{Eq:1by2}
\end{align}
Thus the left inequality in \eqref{comp} is proved with
$c_1=(N Q^{N-1})^{-1}$.

Next, suggest that the right inequality in \eqref{comp} does not hold
in some closed ball $\overline\Bx(g_0,2 r_0)$. Then there are sequences of
points $x_n$, $y_n \in \overline{\Bx}(g_0,r_0)$ converging to the
same point $x_0 \in \overline{\Bx}(g_0, r_0)$, such that
$$
	\varepsilon_n = d_2 (x_n, y_n) \geq n \, d_\infty (x_n, y_n),
$$
where $\varepsilon_n \to 0$ as $n \to \infty$
(otherwise the right inequality in \eqref{comp} would be fulfilled in
$\overline{\Bx}(g_0, r_0)$).
Define on $\overline{\Bx}(g_0,r_0)$
dilations $\mathfrak{D}^g_t$ and $\widehat{\mathfrak{D}}^g_t$
as follows: to an element
$x = \exp(x_N X_N) \circ \dots \circ \exp(x_1 X_1)(g) \in \overline{\Bx}(g_0,r_0)$
assign
$$
	\mathfrak{D}^g_t x
	 = \exp (x_N t^{\deg X_N} X_N)
	 \circ \dots \circ
	 \exp (x_1 t X_1) (g)
$$
and to an element
$\hat x =  \exp(x_N \widehat X^g_N) \circ \dots \circ \exp(x_1 \widehat X^g_1)(g)
\in \overline{\Bx}(g_0,r_0) \cap \mathcal{G}^g$ assign
$$
	\widehat{\mathfrak{D}}^g_t \hat x
	 = \exp (x_N t^{\deg X_N} \widehat X^g_N)
	 \circ \dots \circ
	 \exp (x_1 t \widehat X^g_1) (g).
$$
Observe that $d_2(g, \mathfrak{D}^g_t x) = t d_2(g,x)$ and
$d^g_2(g, \widehat{\mathfrak{D}}^g_t x) = t d^g_2(g, x)$.
Let
$$
	0 < \delta = \sup \{ t > 0 :
	  \mathfrak{D}^g_t x,
	  \widehat{\mathfrak{D}}^g_t x \in \overline{\Bx}(g_0,2 r_0)
	  \text{~for all~} x,g \in \overline{\Bx}(g_0,r_0) \}.
$$
Then $\mathfrak{D}^{x_n}_{\delta / \varepsilon_n} y_n \in \overline{\Bx}(g_0, 2 r_0)$
and
\begin{equation} \label{Eq:D2contra}
	d_2 (x_n, \mathfrak{D}^{x_n}_{\delta / \varepsilon_n} y_n )
	= \frac {\delta} {\varepsilon_n} d_2(x_n, y_n) = \delta > 0.
\end{equation}

Represent $y_n$ in coordinates of the 2nd kind as
$y_n = \exp(y_{nN} X_N) \circ \dots \circ \exp(y_{n1} X_1) (x_n)$
and define
$$
  z_n = \exp(y_{nN} \widehat X^g_N)
  \circ \dots \circ \exp(y_{n1} \widehat X^g_1) (x_n).
$$

Since $d_\infty(x_n, y_n) = d^{x_n}_\infty (x_n, y_n) \leq \frac{\varepsilon_n}{n}$,
from \eqref{Eq:LocalMetricsEquality} it follows
$$
	d^{x_n}_2 (x_n, y_n)
	\leq c^{x_n}_2 d^{x_n}_\infty (x_n, y_n)
	\leq c^{x_n}_2 \frac {\varepsilon_n} {n}
	= O \Bigl(\frac {\varepsilon_n}{n} \Bigr)
$$
where $O(\cdot$) is
uniform in $\overline{\Bx}(g_0,r_0)$. This means that in the representation
$$
	y_n = \exp(v_{nN} \widehat X^g_N)
	\circ \dots \circ
	\exp(v_{n1} \widehat X^g_1) (x_n)
$$
the coordinates $v_j$ meet the property
$|v_{nj}|^{\deg X_j} = O(\frac {\varepsilon_n}{n})$.
Then we can apply Theorem~\ref{Th:PathsEstimate}
to points $y_n$ and $z_n$
and derive that $d^{x_n}_\infty(y_n, z_n) =
o (\frac {\varepsilon_n}{n})$.
Consequently,
$$
	d^{x_n}_\infty(x_n, z_n)
	\leq C (d^{x_n}_\infty(x_n, y_n) + d^{x_n}_\infty(y_n, z_n))
	= O \Bigl( \frac {\varepsilon_n}{n} \Bigr) + o \Bigl( \frac {\varepsilon_n}{n} \Bigr)
	= O \Bigl( \frac {\varepsilon_n}{n} \Bigr).
$$

From Theorem~\ref{Th:PathsEstimate} it also follows
$d^{x_n}_\infty(\mathfrak{D}^{x_n}_{\delta/\varepsilon_n} y_n,
\widehat{\mathfrak{D}}^{x_n}_{\delta/\varepsilon_n} z_n) =
o (\frac {1}{n})$.
Therefore,
\begin{align*}
	 d^{x_n}_2(x_n, \mathfrak{D}^{x_n}_{\delta/\varepsilon_n} y_n)
	& \leq C_1 \bigl(
		d^{x_n}_2(x_n, \widehat{\mathfrak{D}}^{x_n}_{\delta/\varepsilon_n} z_n) +
		d^{x_n}_2(\widehat{\mathfrak{D}}^{x_n}_{\delta/\varepsilon_n} z_n,
			\mathfrak{D}^{x_n}_{\delta/\varepsilon_n} y_n) \bigr) \\
	& = C_1 \Bigl(
		\frac {\delta}{\varepsilon_n} d^{x_n}_2 (x_n, z_n) +
		d^{x_n}_2(\widehat{\mathfrak{D}}^{x_n}_{\delta/\varepsilon_n} z_n,
			\mathfrak{D}^{x_n}_{\delta/\varepsilon_n} y_n) \Bigr) \\
	& = C_2 \Bigl(
		\frac {\delta}{\varepsilon_n} d^{x_n}_\infty (x_n, z_n) +
		d^{x_n}_\infty(\widehat{\mathfrak{D}}^{x_n}_{\delta/\varepsilon_n} z_n,
			\mathfrak{D}^{x_n}_{\delta/\varepsilon_n} y_n) \Bigr) \\
	& = O \Bigl(\frac 1n \Bigr)
	 + o \Bigl( \frac {1}{n} \Bigr)
	 = O \Bigl(\frac 1n \Bigr) \to 0
	 \quad \text{~as~} n \to \infty,
\end{align*}
where $C_1$, $C_2 < \infty$ are bounded,
all $O(\cdot)$ are uniform in $\overline{\Bx}(g_0,r_0)$.

Hence we come to a contradiction with \eqref{Eq:D2contra},
and, therefore, the right inequality in \eqref{comp} is proved.
\end{proof}

\begin{corollary} The quantity $d_2$ is a quasimetric in the sense
of {\rm \cite{Bib:NSW}}, i.~e. the following conditions hold
for the points of the neighborhood $U(g_0)$$:$

\noindent{$1)$} $d_2(u,g)\geq0$, $d_2(u,g)=0$ if and only if $u=g$$;$

\noindent{$2)$} $d_2(u,g)\leq c_1^{-1}c_2d_2(g,u)$, where the constants
$c_1$ and $c_2$ are the ones from the proposition \ref{Prop:MetricsEquality}$;$

\noindent{$3)$} there is a constant $Q_2\geq 1$ such that for
every triple of the points $u$, $w$, $v\in U(g_0)$ we have
$$
    d_2(u,v)\leq Q_2(d_2(u,w)+d_2(w,v)),
$$
where $Q_2=c_1^{-1}c_2Q$ and $Q$ is a constant from the generalized
triangle inequality for $d_\infty$$;$

\noindent{$(4)$}  $d_2(u,v)$ is continuous with respect to the first variable.
\end{corollary}

\begin{proof}
Prove for example the second property:
$d_2(u,g)\leq c_2 d_\infty(u,g)=c_2 d_\infty(g,u)\leq c_1^{-1}c_2d_2(g,u)$.
The third property can be proved using the same procedure.
The last property follows from the continuous dependence of solutions of
ODE on the initial data.
\end{proof}

\subsection{Special coordinate system of the 2nd kind and Rashevsky--Chow Theorem}
\label{Section:PhiCoord}

The goal of this section is to modify the coordinate system
of the 2nd kind
$$
  (t_1, \dots, t_N) \mapsto \exp (t_N X_N) \circ \dots
    \circ \exp (t_1 X_1) (g)
$$
in the following way. We prove that exponents of nonhorizontal
vector fields $X_k$, $k = \dim H_1+1, \dots, N$, can be replaced
by compositions of exponents of horizontal vector fields
$X_1, \dots, X_{\dim H_1}$ and the resulting mapping still
covers a neighborhood of $g$. For Carnot groups this property
is known as the following statement.

\begin{lemma}[\cite{Bib:Folland-Stein}] \label{Lemma:FSRepresentation}
Let $\mathbb{G} = (\mathbb{R}^N, \cdot)$ be a Carnot group and let
vector fields $Y_1, \dots, Y_n$ be the basis of horizontal subspace
$V_1$ of its Lie algebra. Then every point $v \in \mathbb{G}$ can be
represented as
$$
  v = \prod_{k=1}^L \exp ( a_k Y_{i_k} )(0)
$$
where $1 \leq i_k \leq n$,
$|a_k| \leq c_1 \Vert v \Vert_\infty$, constants $L$ and $c_1$
are independent of $v$.
\end{lemma}

\begin{lemma}
Fix $g \in \mathcal{M}$.
There exists mapping
$\widehat \Phi_g : B_e(0, \varepsilon) \to \mathcal{G}^g$
defined as
\begin{multline}
  \widehat \Phi_g :
  (t_1, \dots, t_N) \mapsto
  \widehat \Phi_N (t_N) \circ \dots \circ
  \widehat \Phi_{\dim H_1 + 1} (t_{\dim H_1 + 1}) \\
  \circ \exp ( \widehat X^g_{\dim H_1}) \circ \dots \circ
  \exp ( \widehat X^g_1 ) (g) \label{Eq:PhiHatSystem}
\end{multline}
that is a homeomorphism of a ball $B_e(0, \varepsilon)$ onto the
neighborhood $V(g) \subset \mathcal{G}^g$ of a point $g$ with the
mappings $\widehat \Phi_k$ enjoying
$$
  \widehat \Phi_k (t)(\cdot) =
  \begin{cases}
    \exp (a_{L,k} t \widehat X^g_{L,k}) \circ \dots \circ
    \exp (a_{1,k} t \widehat X^g_{1,k}) (\cdot), & t \geq 0, \\
    \exp (a_{1,k} t \widehat X^g_{1,k}) \circ \dots \circ
    \exp (a_{L,k} t \widehat X^g_{L,k}) (\cdot), & t < 0, \\
  \end{cases}
$$
where $|a_{i,k}| \leq c_1$
for all $k = \dim H_1 + 1,\dots, N$, $i = 1,\dots,L$,
every $\widehat X^g_{i,k}$ is from
$\{ \widehat X^g_1, \dots, \widehat X^g_{\dim H_1} \}$.
\end{lemma}

\begin{proof}
Consider coordinate system of the 2nd kind on the
nilpotent tangent cone~$\mathbb{G}_g \mathcal{M}$.
$$
 \Theta_g (t_1, \dots, t_N)
 = \exp ( t_N (\widehat X^g_N)' ) \circ \dots
   \circ \exp ( t_1 (\widehat X^g_1)' )(0).
$$
The mapping $\Theta_g$ is a diffeomorphism of $\mathbb{R}^N$.
For every nonhorizontal vector field $(\widehat X^g_k)'$ fix
decomposition given by Lemma~\ref{Lemma:FSRepresentation}
$$
 \exp ( (\widehat X^g_k)' )(0)
 = \exp ( a_{L,k} (\widehat X^g_{L,k})' ) \circ \dots
 \circ \exp ( a_{1,k} (\widehat X^g_{1,k})' ) (0).
$$
Here $|a_{i,k}| < c_1$ for all $i=1,\dots,L$,
$k = \dim H_1 + 1, \dots, N$, and every
$(\widehat X^g_{i,k})'$ is from the set
$\{ (\widehat X^g_1)', \dots, (\widehat X^g_{\dim H_1})' \}$.
Applying dilation $\delta^g$ to this decomposition we obtain
the following representation
\begin{align}
  \delta^g_t \exp ( (\widehat X^g_k)' )(0)
  & = \exp ( t^{\deg X_k} (\widehat X^g_k)' )(0) \notag \\
  & = \exp ( a_{L,k} t (\widehat X^g_{L,k})' ) \circ \dots
    \circ \exp ( a_{1,k} t (\widehat X^g_{1,k})' ) (0),
    & t \geq 0, \notag \\
  \delta^g_t \exp ( (\widehat X^g_k)' )(0)
  & = \exp ( - |t|^{\deg X_k} (\widehat X^g_k)' )(0) \notag \\
  & = \exp ( a_{1,k} t (\widehat X^g_{1,k})' ) \circ \dots
    \circ \exp ( a_{L,k} t (\widehat X^g_{L,k})' ) (0),
    &  t < 0. \label{Eq:PhiHatLine}
\end{align}
Since vector fields $(\widehat X^g_k)'$ are left-invariant,
representation~\eqref{Eq:PhiHatLine} holds also if we replace
$0$ by arbitrary $x \in \mathbb{G}_g \mathbb{M}$.

Next, we push-forward representation~\eqref{Eq:PhiHatLine}
using local group isomorphism~$\theta_g$.
Define mappings
$\widehat \Phi_k :
[-\varepsilon, \varepsilon] \times
\mathrm{Box}(g, \varepsilon)
\to \mathcal{G}^g$ as
\begin{equation} 
  \widehat \Phi_k (t)(w) =
  \begin{cases}
   \exp ( a_{L,k} t \widehat X^g_{L,k} ) \circ \dots
    \circ \exp ( a_{1,k} t \widehat X^g_{1,k} ) (w), & t \geq 0, \\
   \exp ( a_{1,k} t \widehat X^g_{1,k} ) \circ \dots
    \circ \exp ( a_{L,k} t \widehat X^g_{L,k} ) (w), & t < 0
  \end{cases}
  \label{Eq:PushPhi}
\end{equation}
where, by definition,
$$
  \exp (a \widehat X^g_i) \circ \exp (b \widehat X^g_j) =
  \theta_g \circ \exp (a (\widehat X^g_i)') \circ
  \exp (b (\widehat X^g_j)') \circ \theta_g^{-1}
$$
and $\varepsilon > 0$ is small enough that~\eqref{Eq:PushPhi} makes
sense for all $k = \dim H_1 + 1, \dots, N$,
$t \in [-\varepsilon, \varepsilon]$ and
$w \in \mathrm{Box}(g, \varepsilon)$.

Consider a mapping $\widehat \Phi_g$ defined as in
\eqref{Eq:PhiHatSystem}.
Since, by construction,
$$
  \widehat \Phi_g(t_1, \dots, t_N) =
  \theta_g \circ \Theta_g (t_1^{\deg X_1}, \dots, t_N^{\deg X_N}),
$$
the mapping $\widehat \Phi_g$ is a homeomorphism of a
ball $B_e(0,\varepsilon) \subset \mathbb{R}^N$ onto the neighborhood
$V(g) \subset \mathcal{M} \cap \mathcal{G}^g$.
\end{proof}

For every point $g \in U(g_0)$ define mappings
$\Phi_k: [-\varepsilon, \varepsilon] \to \mathcal{M}$
as
\begin{equation} 
  \Phi_k(t)(\cdot) =
  \begin{cases}
  \exp ( a_{L,k} t X_{L,k} ) \circ \dots
    \circ \exp ( a_{1,k} t X_{1,k} ) (\cdot), & t \geq 0,  \\
  \exp ( a_{1,k} t X_{1,k} ) \circ \dots
    \circ \exp ( a_{L,k} t X_{L,k} ) (\cdot), & t < 0,
  \end{cases}
  \label{Eq:CoordPhiLine}
\end{equation}
where coefficients $a_{i,k}$, $i = 1,\dots,L$,
$k = \dim H_1+1, \dots, N$, are taken from
the representation \eqref{Eq:PhiHatLine}.
Define also a mapping
$\Phi_g : B_e(0, \varepsilon) \to \mathcal{M}$ as
\begin{multline} \label{Eq:CoordPhiSystem}
  \Phi_g : (t_1, \dots, t_N) \mapsto \Phi_N(t_N) \circ \dots
    \circ \Phi_{\dim H_1 + 1}(t_{\dim H_1 + 1}) \\
  \circ \exp (t_{\dim H_1} X_{\dim H_1}) \circ \dots \circ \exp(t_1 X_1)(g).
\end{multline}
Next, we prove that $\Phi_g$ is the desired mapping, i.~e.
there is a neighborhood $V(g)$ such that
$V(g) \subset \Phi(B_e(0, \varepsilon))$.

\begin{theorem} \label{Theorem:CoordSys}
  Fix the point $g_0 \in \mathcal{M}$. Let
  $X_1, \dots, X_{\dim H_1}$ be a basis in $H_1$. Then there is a
  neighborhood $U(g_0)$ such that for every point $g\in U(g_0)$
  an element $v\in U(g_0)$ can be represented as
  \begin{equation} \label{Eq:HorizPathForm}
    v = \exp(a_L X_{j_L}) \circ \dots 
    \circ \exp(a_2 X_{j_2}) \circ \exp(a_1 X_{j_1}) (g),
  \end{equation}
  where $1 \leq j_i \leq \dim H_1$, $i=1,\dots,L$, $L\in \mathbb N$,
  $|a_i|\leq c_2 \, d_\infty(g,v)$, constants $L$ and $c_2$ are independent
  of $g$ and $v$.
\end{theorem}

\begin{proof}
Fix $g_0 \in \mathcal{M}$.
Let $\widehat \Phi_k(t)(\cdot)$ and $\Phi_k(t)(\cdot)$ be defined as
in \eqref{Eq:PhiHatLine} and \eqref{Eq:CoordPhiLine}.
By Theorem~\ref{Th:PathsEstimate} we have
$$
  d_\infty \bigl( \widehat \Phi_k(t)(w),
  \Phi_k(t)(w) \bigr) = o(t)
  \quad \text{as } t \to 0
$$
where $o(t)$ is uniform with respect to $g$, $w$
in a compact neighborhood $U(g_0)$.

Let $B_e(0, r)$ be an Euclidean ball in $\mathbb{R}^N$
and mappings $\widehat \Phi_g$ and $\Phi_g : B_e(0, r) \to \mathcal{M}$
be defined as in \eqref{Eq:PhiHatSystem} and \eqref{Eq:CoordPhiSystem}.
Observe that both mappings are continuous and that
$d_\infty ( \Phi_g(x), \widehat \Phi_g(x)) = o(r)$ as $r \to 0$
where $o(r)$ is uniform in $g \in U(g_0)$ and $x \in B_e(0, r)$.
Moreover, $\widehat \Phi_g$ is a homeomorphism of $B_e(0, r)$ onto a
neighborhood $V(g) \in \mathcal{M} \cap \mathcal{G}^g$.

Define $\psi = \Phi_g \circ \widehat \Phi_g^{-1}$.
The mapping $\psi : V(g) \to \mathcal{M}$
is continuous and $d_{\infty}(v, \psi(v)) = o(d_\infty(g,v))$
as $v \to g$ where $o(\cdot)$ is uniform in $g, v \in U(g_0)$.
Choose $\varepsilon_0 > 0$ such that
$d_\infty(v, \psi(v)) \leq \frac{\varepsilon}{2Q}$ for every
$v \in \overline{\mathrm{Box}(g,\varepsilon)}$,
$0 < \varepsilon \leq \varepsilon_0$
and $g \in U(g_0)$,
where $Q \geq 1$ is a constant from the generalized triangle
inequality for $d_\infty$.
Next, we prove that $\psi(\mathrm{Box}(g,\varepsilon))$ is a
neighborhood of $g$.

Consider a homotopy $\psi_t (v) = \delta^v_{1-t} \psi(v)$,
$t \in [0,1]$. It is clear that $\psi_0(v) = \psi(v)$ and
$\psi_1(v) = v$. Fix a point
$w \in \mathrm{Box}(g, \frac{\varepsilon}{2Q})$.
Then for every $v \in \partial \mathrm{Box}(g, \varepsilon)$ we have
$$
  \varepsilon = d_\infty(g, v) \leq Q \bigl(
  d_\infty(g, w) + d_\infty(w, v) \bigr)
  < \frac{\varepsilon}{2} + Q d_\infty(w, v).
$$
Hence, $d_\infty(w, v) > \frac{\varepsilon}{2Q}$. On the other
side, for all $v \in \partial \mathrm{Box}(g, \varepsilon)$ we also have
\begin{align*}
  d_\infty( \psi_t(v), v)
  & = d_\infty( \delta^v_{1-t} \psi(v), v ) \\
  & = d^v_\infty( \delta^v_{1-t} \psi(v), v )
    = (1-t) d^v_\infty( \psi(v), v ) \\
  & \leq d^v_\infty( \psi(v), v)
    = d_\infty( \psi(v), v)
    \leq \frac {\varepsilon} {2Q}.
\end{align*}
Consequently,
$w \not\in \psi \bigl(
\partial \mathrm{Box} (g, \varepsilon) \bigr)$
for all $t \in [0, 1]$.
Therefore, the topological degree of $\psi_t$ at $w$ is invariant
for all $t \in [0, 1]$. Since
$$
  \deg (w, \mathrm{Box}(g,\varepsilon), \psi)
  = \deg (w, \mathrm{Box}(g,\varepsilon), \psi_1)
  = \deg (w, \mathrm{Box}(g,\varepsilon), \psi_0) = 1,
$$
we conclude $w \in \psi \bigl( \mathrm{Box}(g,\varepsilon) \bigr)$.
In other words
$\mathrm{Box}(g, \frac{\varepsilon}{2Q}) \subset
\Phi_g ( \mathrm{Box}_e(0, \varepsilon) )$, where
$\mathrm{Box}_e(0, \varepsilon) =
\{ x \in \mathbb{R}^N : |x_i| < \varepsilon, i=1,\dots,N \}$ is an Euclidean cube.

Let $U(g_0)$ be a neighborhood of $g_0$ small enough that
$$
    U(g_0)\subset\bigcap\limits_{g \in U(g_0)}
    \mathrm{Box} (g, \tfrac{\varepsilon_0}{2Q}).
$$
Let $\varepsilon = d_\infty(g,v)$ where $g, v \in U(g_0)$.
Then there exists a tuple of numbers
$(t_1, \dots, t_N)$ such that $|t_i| < 2Q\varepsilon$ and
$v = \Phi_g(t_1, \dots, t_N)$.
This completes the proof.
\end{proof}

An absolutely continuous curve $\gamma : [0,T] \to \mathcal{M}$
is said to be {\it horizontal} if $\dot\gamma(t) \in H_{\gamma(t)} \mathcal{M}$
for almost all $t \in [0,T]$.

As an immediate consequence of Theorem~\ref{Theorem:CoordSys} we obtain
the following generalization of
Rashevsky--Chow theorem \cite{Bib:Rashevsky, Bib:Chow, Bib:Karm-Vod}.
For $C^1$-smooth fields $X_1,\ldots,X_N$
this statement is new.

\begin{theorem}
\label{Th:Rashevsky}
$1)$ Let $g \in \mathcal{M}$. There exists a neighborhood~$U$ of a point~$g$
such that every pair of points $u$, $v \in U$ in a Carnot--Carath\'eodory
space~$\mathcal{M}$ can be joined by an absolutely continuous
horizontal curve~$\gamma$ constituted of at most~$L$ segments of integral
lines of basic horizontal fields where~$L$ is independent of the choice of
points $x$, $y \in U$.

$2)$ Every pair of points $u$, $v$ in a connected Carnot--Carath\'eodory
space~$\mathcal{M}$ can be joined by an absolutely continuous
horizontal curve~$\gamma$ constituted of finite number of segments of
integral lines of basic horizontal fields.
\end{theorem}

\subsection{Carnot--Carath\'eodory metric and Ball-Box Theorem}

The {\it Carnot--Carath\'eodory distance} between two points $x$,
$y \in \mathcal{M}$  is defined as
\begin{multline*}
	d_{cc}(x,y) = \inf \{ T > 0 : \text{there exists a horizontal path~}
	\gamma : [0,T] \to \mathcal{M}, \\
	\gamma(0) = x, \gamma(T) = y, |\dot\gamma(t)| \leq 1 \}.
\end{multline*}
Theorem~\ref{Th:Rashevsky} guarantees that $d_{cc}(x,y) < \infty$
for all $x$, $y \in \mathcal{M}$.
An open ball in Carnot--Carath\'eodory metric with center in $x$ and
radius $r$ we denote as $B_{cc}(x,r)$.

The following statement is called the local approximation theorem.
It was formulated in \cite[p.~135]{Bib:Gromov} for
``sufficiently smooth vector fields''.
It was proved in \cite{Bib:VK-Approx} for $C^{1,\alpha}$-smooth vector
fields but the same arguments work for the case of $C^1$-smooth vector
fields since they are based on the property \eqref{Eq:EpsFieldExpansion}
\cite[Theorem~7]{Bib:Karm-Vod1}.
\begin{theorem}[\cite{Bib:VK-Approx, Bib:Karm-Vod1}]
\label{Th:LocApprForDcc}
Let $g \in \mathcal{M}$. Then for every two points $u$,
$v \in B_{cc}(g,\varepsilon)$ we have
$$
	|d_{cc}(u,v) - d^g_{cc}(u,v)| = o(\varepsilon)
	\quad \text{~as~} \varepsilon \to 0
$$
where $o(\varepsilon)$ is uniform in $u$, $v \in B(g,\varepsilon)$
and $g$ belonging to some compact set.
\end{theorem}

As a corollary we obtain a comparison of metric $d_{cc}$ and quasimetric
$d_\infty$ and Ball-Box theorem.

\begin{theorem}[{\cite[Theorem 11]{Bib:Karm-Vod1}}]
\label{Th:CmpDinftyAndDcc}
Let $g \in \mathcal{M}$. There exists a compact neighborhood
$U(g) \subset \mathcal{M}$ and constants $0 < C_1 \leq C_2 < \infty$
independent of $u$, $v \in U(g)$ such that
\begin{equation} \label{CmpDinftyAndDcc}
	C_1 d_\infty(u,v) \leq d_{cc}(u,v) \leq C_2 d_\infty(u,v)
\end{equation}
for all $u$, $v \in U(g)$.
\end{theorem}

The following statement was proved under smooth enough vector fields in
\cite{Bib:NSW, Bib:Gromov}, for $C^{1,\alpha}$-smooth vector fields,
$\alpha \in (0,1]$, in \cite{Bib:Karm-Vod} and for $C^1$-smooth
vector fields in \cite{Bib:Karm-Vod1}.

\begin{corollary}[Ball-Box theorem \cite{Bib:Karm-Vod1}]
Given a compact neighborhood $U \in \mathcal{M}$,
there exist constants $0 < C_1 \leq C_2 < \infty$ and $r_0 > 0$
independent of $x \in U$ such that
$$
	\Bx(x,C_1 r) \subset B_{cc}(x,r) \subset \Bx(x,C_2 r)
$$
for all $r \in (0, r_0)$ and $x \in U$.
\end{corollary}

\section{Approximate limit and differentiability}
\label{Section2}

\subsection{Hausdorff measure}

The (spherical) $k$-dimensional {\it Hausdorff measure} of the set $E$ with
respect to metric $d_{cc}$ is the quantity
\begin{equation*}
	\mathcal{H}^k(E) = \lim_{\varepsilon \to 0+} \inf \Bigl\{
	\sum_i r^k_i : E \subset \bigcup_i B_{cc}(x_i, r_i), r_i < \varepsilon
	\Bigr\}.
\end{equation*}

\begin{theorem}[\cite{Bib:Mitchell, Bib:Karm-Vod}]
The Hausdorff dimension of $\mathcal{M}$ with respect to $d_{cc}$
is equal to
$$
	\nu = \sum_{k=1}^N \deg X_k = \sum_{i=1}^M i (\dim H_i - \dim H_{i-1})
$$
where $\dim H_0 = 0$.
\end{theorem}

Ball-Box theorem implies the double property of measure.

\begin{proposition}
We have
$$
	\mathcal{H}^\nu(B_{cc}(x,2r)) \leq C \mathcal{H}^\nu(B_{cc}(x,r))
$$
where $C < \infty$ is bounded in $r \in (0, r_0]$ and $x$ belonging
to some compact part $V \subset \mathcal{M}$.
\end{proposition}

\subsection{Approximate limit and its properties}
\label{Section:ApproxLimit}
The {\it density} of a set $Y$ at $x \in \mathcal{M}$ is a limit
$$
	\lim_{r \to +0} \frac {\mathcal{H}^\nu( B_{cc}(x,r) \cap Y )} {\mathcal{H}^\nu( B_{cc}(x,r) )},
$$
if it exists at $x$ (where $\nu$ is the Hausdorff dimension of the space $\mathcal{M}$).

Let $E \subset \mathcal{M}$ be a measurable set and $f : E \to \mathbb M$ be a mapping to
a metric space $\mathbb M$.

A point $y \in \mathbb M$ is called the {\it approximate limit} of the mapping $f$
at the point $g \in E$ of density $1$ and is denoted by
$y = \ap \lim\limits_{x \to g} f(x)$ if the density
of set $E \setminus f^{-1}(W)$ at $g$ equals zero for every
neighborhood $W$ of the point $y$.

In the case $\mathbb M = \overline{\mathbb R}$ we also define the {\it approximate
upper limit} of the function $f$ at the point $g \in E$, denoted by
$\ap \varlimsup\limits_{x \to g} f(x)$, as the greatest lower bound of
the set of all numbers $s$ for which the density of the set
$\{ z \in \mathcal{M} : f(z) > s \}$ at the point $g$ equals zero.
By definition,
$\ap \varliminf\limits_{x \to g} f(x) = -\ap \varlimsup\limits_{x \to g} (-f(x))$ is the
{\it approximate lower limit}. It is easy to verify that
$\ap \varliminf\limits_{x \to g} f(x) \leq \ap \varlimsup\limits_{x \to g} f(x)$ and that
$\ap \lim\limits_{x \to g} f(x)$ exists if and only if
$\ap \varliminf\limits_{x \to g} f(x) = \ap \varlimsup\limits_{x \to g} f(x)$.

Next we state several properties regarding measurability and approximate limit
which we need in further arguments.

\begin{property} \label{Property1}
    Let $S$ be a $\mathcal{H}^\nu \times \mathcal{H}^{\tilde\nu}$-measurable set
    in $\mathcal{M} \times \widetilde{\mathcal{M}}$ and $z_0$ be
    a fixed point in $\widetilde{\mathcal{M}}$.
    For every $\varepsilon > 0$ and $\delta > 0$
    define $T$ as a set of the points $x$ for which
    $$
    \mathcal{H}^{\tilde\nu} \{ z: (x,z) \in S, \widetilde d_{cc}(z_0,z) \leq r \}
        \leq \varepsilon r^{\tilde\nu} \quad
            \text{~for all~} \: 0 < r < \delta.
    $$
    Then the set $T$ is measurable.
\end{property}

Really, for any $r > 0$, a set
$$
    S_r = S \cap \{ (x,z) : \widetilde d_{cc}(z_0,z) \leq r \}
    = S \cap ( \mathcal{M} \times \overline{\widetilde B_{cc}(z_0, r)} )
$$
is $\mathcal{H}^\nu \times \mathcal{H}^{\tilde\nu}$-measurable. By
Tonelli--Fubini theorem the set $\{ z: (x, z) \in S_r \}$ is $\mathcal{H}^{\tilde\nu}$-measurable
for $\mathcal{H}^\nu$-almost all $x$ and
$$
  \iint\limits_{\mathcal{M} \times \mathcal{\widetilde{M}}} \chi_{S_r} (x,z) \, dx \, dz
  = \int\limits_{\mathcal{M}} \int\limits_{\mathcal{\widetilde{M}}} \chi_{S_r} (x,z) \, dz \, dx
  = \int\limits_{\mathcal{M}} \mathcal{H}^{\tilde\nu} \{ z: (x, z) \in S_r \} \, dx.
$$
Consequently, the mapping
$$
  \varphi : x \mapsto \int\limits_{\mathcal{\widetilde{M}}} \chi_{S_r} (x,z) \, dz
  = \mathcal{H}^{\tilde\nu} \{ z: (x, z) \in S_r \}
$$
is $\mathcal{H}^\nu$-measurable. Then we have
$$
    T = \bigcap_{r \in (0,\delta) \cap \mathbb{Q}}
        \{ x: \varphi(x) \leq \varepsilon r^{\nu} \},
$$
where $\mathbb{Q}$ denotes the set of rational numbers.
It remains only to note that every set
$\{ x: \varphi(x) \leq \varepsilon r^{\nu} \}$
is $\mathcal{H}^\nu$-measurable.

\begin{property} \label{Property2}
    If $\sigma : \mathcal{M} \times \widetilde{\mathcal{M}} \to\overline{\mathbb R}$ is
    $\mathcal{H}^\nu \times \mathcal{H}^{\tilde\nu}$-measurable real-valued mapping and
    $z_0$ is a point in $\widetilde{\mathcal M}$ then
    $$
        \ap \varlimsup_{z \to z_0} \sigma(x, z) \; \text{~and~} \;
            \ap \varliminf_{z \to z_0} \sigma(x, z)
    $$
    are $\mathcal{H}^\nu$-measurable mappings of argument $x$.
\end{property}

First, notice that
$$
    \{ x \in \mathcal{M} : \ap \varlimsup_{z \to z_0} \sigma(x, z) \leq \tau \}
        = \bigcap_{t > \tau} A_t = \bigcap_{n=1}^{\infty} A_{\tau+\frac{1}{n}},
$$
where $A_t$ is a set of the points $x \in \mathcal{M}$ for which the set
$\{ z \in \widetilde{\mathcal M} : \sigma(x, z) > t \}$ has the density zero at $z_0$.
We have to make sure that $A_t$ is measurable. In order to do this
we apply Property~\ref{Property1} to the set
$$
    S_t = \{ (x, z) \in \mathcal{M} \times \widetilde{\mathcal{M}} : \sigma(x, z) > t \}
$$
and derive that the set $T_t(m,k)$ of the points $x \in \mathcal{M}$ for which
$$
    \mathcal{H}^{\nu} \{ z: (x,z)\in S_t,\, \widetilde d_{cc}(z_0,z) \leq r \} \leq \frac{r^\nu}{m}
         \ \text{~for all~} 0 < r < k^{-1},
$$
is measurable for all positive integers $m$ and $k$.
It remains only to observe that
$$A_t = \bigcap\limits_{m=1}^{\infty} \bigcap\limits_{k=1}^{\infty} T_t(m,k).$$

\subsection{Differentiability in the sub-Riemannian geometry}
\label{Section:Differentiability}

Fix $E \subset \mathbb{R}$ and a limit point $s \in E$. The mapping
$\gamma : E \to \mathcal{M}$ has {\it sub-Riemannian derivative} at
the point $s$ if there is an element $a \in \mathcal{G}^{\gamma(s)}$ such that
\begin{equation} \label{Def:Deriv}
    d_{cc}^{\gamma(s)} ( \gamma(s+t), \delta^{\gamma(s)}_t a ) = o(t) \:
        \text{~as~} t \to 0, \ s+t \in E.
\end{equation}
We use the notation
$a = \frac d{dt}_{sub} \gamma(t+s) \vert_{t = 0}$.
A derivative is called {\it horizontal} in the case
$a \in \exp (H_{\gamma(s)} \mathcal{M})$, i.~e.
$$
	a = \exp \Bigl( \sum_{j=1}^{\dim H_1} \alpha_j \widehat X^{\gamma(s)}_j \Bigr)(\gamma(s))
	  = \exp \Bigl( \sum_{j=1}^{\dim H_1} \alpha_j X_j \Bigr)(\gamma(s))
$$
for certain $\alpha_j \in \mathbb{R}$. 

In \cite{Bib:Vod-Spaces} it is proved that for a curve in the
Carnot--Carath\'eodory space to be horizontally differentiable it is
sufficient to be a Lipschitz mapping. Recall that
$\gamma : E \subset \mathbb{R} \to \mathcal{M}$ is called a {\it Lipschitz} mapping
if there is a constant $C > 0$ such that the inequality
$$
	d_{cc}( \gamma(x), \gamma(y) ) \leq C |x-y|
$$
holds for all $x, y \in E$.
\begin{theorem}[\cite{Bib:Vod-Spaces}]\label{Theorem:DifLipschitz}
    Every Lipschitz mapping $\gamma : E \to \mathcal{M}$,
    where the set $E \subset \mathbb{R}$ is closed,
    has horizontal derivative almost everywhere in $E$.
\end{theorem}

The mapping $f : E \subset \mathcal{M} \to \widetilde{\mathcal{M}}$
of two Carnot--Carath\'eodory spaces
is called \cite{Bib:Vod-Greshnov} differentiable at the point $g \in E$
if there is horizontal homomorphism $L : \mathcal{G}^g \to \mathcal{G}^{f(g)}$ of the
local Carnot groups such that
\begin{equation} \label{Eq:Differential}
    \widetilde d_{cc}^{f(g)}( f(v), L(v) )
    = o( d_{cc}^g(g, v) ) \text{~as~} E \cap \mathcal{G}^g \owns v \to g.
\end{equation}
Recall that {\it the horizontal homomorphism} of Carnot groups is a
homomorphism $L : {\mathbb G} \to \widetilde {\mathbb G}$ such that
$DL(0)(H {\mathbb G}) \subset H \widetilde {\mathbb G}$.

Local approximation theorem (Theorem~\ref{Th:LocApprForDcc}) gives
an opportunity to use both metrics of the initial space and of
local Carnot group in the definition~\eqref{Eq:Differential}.
Indeed, since
$$
 \widetilde d_{cc} (f(g), f(v))
 \leq \widetilde d_{cc} (f(g), L(v))
 + \widetilde d_{cc} (L(v), f(v)),
$$
we have
\begin{align}
  \widetilde d_{cc} (f(v), L(v))
  & = \widetilde d^{f(g)}_{cc}(f(v), L(v))
    + o \bigl( \widetilde d^{f(g)}_{cc} (f(g), f(v)) \bigr)
    + o \bigl( \widetilde d^{f(g)}_{cc} (f(g), L(v)) \bigr) \notag \\
  & = \widetilde d^{f(g)}_{cc}(f(v), L(v))
    + o \bigl( \widetilde d^{f(g)}_{cc} (f(v), L(v)) \bigr)
    + o \bigl( \widetilde d^{f(g)}_{cc} (f(g), L(v)) \bigr) \notag \\
  & = o( d^g_{cc}(g,v) )
    + o \Bigl( d^g_{cc} (v,g) \sup_{u: \: d^g_{cc}(u,g) = 1}
      \widetilde d^{f(g)}_{cc} (f(g), L(u)) \Bigr) \notag \\
  & = o( d^g_{cc}(g,v) ) = o( d_{cc}(g,v) ). \label{Eq:CDifferential}
\end{align}

The homomorphism $L : \mathcal{G}^g \to \mathcal{G}^{f(g)}$ satisfying
\eqref{Eq:Differential} is called {\it the differential} of the mapping
$f$ and is denoted by $D_g f$. One can show that if $g$ is the density point
then the differential is unique.
Moreover, it is easy to verify that differential commutes
with the one-parameter dilation group:
\begin{equation} \label{Eq:DiffCommutes}
    \tilde \delta^{f(g)}_t \circ D_g f = D_g f \circ \delta^g_t.
\end{equation}

If $v \in \mathcal{G}^g$ and $\delta^g_t v \in \mathcal{G}^g$ then, by \eqref{Eq:DiffCommutes},
we have
\begin{multline} \label{Eq:DiffToDeriv}
        \widetilde d_{cc}^{f(g)}( f(\delta^g_t v), \tilde \delta^{f(g)}_t D_g f(v) )
            = \widetilde d_{cc}^{f(g)}( f(\delta^g_t v), D_g f(\delta^g_t v) ) \\
        = o( d_{cc}^g(g, \delta^g_t v) ) = d_{cc}^g(g,v) o(t),
\end{multline}
i.~e. element $D_gf(v)$ is a derivative
of the curve $\gamma(t) = f(\delta^g_t v)$ at $t=0$.

By {\it the derivative} of the mapping $f$ along the horizontal
vector field $X$ at the point $g$ we mean the derivative of the curve
$$
    \gamma(t) = f(\delta^g_t \exp \widehat X^{g}(g)) = f(\exp t X(g))
$$
for $t=0$. We use the notation $Xf(g)$ to denote this derivative.
To be more precise we have to write $\widetilde \exp Xf(g)$ since usually
$Xf(g)$ is the Riemannian derivative 
$\left. \frac{d}{dt} f( \exp(tX)(g) )\right|_{t=0}$. To simplify notations
we will use $Xf(g)$ for the sub-Riemannian derivative except of the cases when
the opposite is stated explicitly.

The mapping $f : E \subset \mathcal{M} \to \widetilde{\mathcal{M}}$
of two Carnot--Carath\'eodory spaces is called a {\it Lipschitz} mapping if
there is a constant $C > 0$ such that the inequality
$$
	\widetilde d_{cc}( f(x), f(y) ) \leq C d_{cc} ( x, y )
$$
holds for all $x, y \in E$.

In the work \cite{Bib:Vod-Spaces} there were generalized the classical
Rademacher \cite{Bib:Rademacher} and Stepanoff \cite{Bib:StepDiff}
theorems to the case of Carnot--Carath\'eodory spaces.
\begin{theorem}[{\cite[Theorem 4.1]{Bib:Vod-Spaces}}] \label{Theorem:RademacherLike}
    Let $E$ be a set in $\mathcal{M}$ and let
    $f : E \to \widetilde{\mathcal{M}}$ be a Lipschitz
    mapping. Then $f$ is differentiable almost everywhere in $E$
    and the differential is unique.
\end{theorem}

\begin{theorem}[{\cite[Theorem 5.1]{Bib:Vod-Spaces}}] \label{Theorem:StepanoffLike}
    Let $E$ be a set in $\mathcal{M}$ and let a mapping
    $f : E \to \widetilde{\mathcal{M}}$
    satisfy the condition
    $$
        \varlimsup_{x \to a, x \in E}
            \frac {\widetilde d_{cc}(f(a),f(x))} {d_{cc}(a,x)} < \infty
    $$
    for almost all $a \in E$. Then $f$ is differentiable
    almost everywhere in $E$ and the differential is unique.
\end{theorem}

Here we will write an alternative proof of Theorems~\ref{Theorem:RademacherLike}
and \ref{Theorem:StepanoffLike} using the theorem on approximate differentiability.

\subsection{Approximate differentiability}
\label{Section:ApproximateDiff}

Now we replace a regular limit in \eqref{Def:Deriv} by the approximate one.
This leads us to definition of an approximate (horizontal)
derivative as an element $a \in \exp H \mathcal{G}^{\gamma(s)}$ such that
$$
    \ap \lim_{t \to 0} \frac {d_{cc}^{\gamma(s)} ( \gamma(s+t), \delta^{\gamma(s)}_t a )} {|t|} = 0,
$$
i.~e. the set
$$
    \{ t \in (-r,r) : d_{cc}^{\gamma(s)} ( \gamma(s+t), \delta^{\gamma(s)}_t a ) > |t| \varepsilon \}
$$
has density zero at the point $t=0$ for an arbitrary $\varepsilon > 0$.

Similarly an approximate differential is the horizontal homomorphism
$L : \mathcal{G}^g \to \mathcal{G}^{f(g)}$ of the local Carnot groups such that
$$
    \ap \lim_{v \to g} \frac {\widetilde d_{cc}^{f(g)}(f(v), L(v))} {d_{cc}^g(g,v)} = 0,
$$
i.~e. a set
$$
    \{ v \in B_{cc}(g,r) \cap \mathcal{G}^g : \widetilde d_{cc}^{f(g)}( f(v), L(v) ) > d_{cc}^g(g,v) \varepsilon \}
$$
has $\mathcal{H}^\nu$-density zero at the point $v = g$ for any $\varepsilon > 0$.
We denote such homomorphism as $\ap D_g f$.

Using the notion of an approximate differential we can generalize
Theorem~\ref{Theorem:StepanoffLike} in the following direction.
\begin{theorem} \label{Theorem:ApprStepanoffLike}
    Let $E$ be a set in $\mathcal{M}$ and let
    $f : E \to \widetilde{\mathcal{M}}$ meet the condition
    \begin{equation} \label{Eq:TotalBound}
        \ap \varlimsup_{x\to g}
        \frac {\widetilde d_{cc}(f(g),f(x))} {d_{cc}(g,x)} < \infty.
    \end{equation}
    Then $f$ is approximately differentiable
    almost everywhere in $E$.
\end{theorem}

For proving Theorem~\ref{Theorem:ApprStepanoffLike} we need the
following statement.
\begin{theorem} \label{Theorem:SetFamilyLipschitz}
    Let $E$ be a measurable subset in $\mathcal{M}$ and
    $f:E \to \widetilde{\mathcal{M}}$ be a measurable mapping enjoying
    \eqref{Eq:TotalBound} for all points $g \in E$. Then there is
    a sequence of disjoint sets $E_0, E_1, \dots$, such that
    $E = E_0 \cup \bigcup\limits_{i=1}^{\infty} E_i$, $\mathcal{H}^\nu(E_0) = 0$
    and every restriction $f|_{E_i}$, $i=1,2,\dots$, is a Lipschitz mapping.
\end{theorem}

\begin{proof}
Since our considerations are local, we limit our arguments to the case
when $E \subset U$ where $U$ is an open subset in $\mathcal M$.
Consider a sequence of 
sets
$$
	U_m = \{ x \in U : d_{cc}(x , \partial U ) \geq 2 m^{-1} \},
	\quad m \in \mathbb{N}.
$$

Each $U_m$ is closed and $\bigcup\limits_{m=1}^{\infty} U_m = U$.
For all distinct points $u$ and $v$ of $U$ the relation
$$
   h(u,v) = \frac{ \mathcal{H}^\nu( B_{cc}(u,d_{cc}(u, v)) \cap B_{cc}(v, d_{cc}(u, v)) ) }
     { d_{cc}(u, v)^\nu }, \quad u \ne v,
$$
is a continuous real-valued function. For every $m$ define a constant
$$
	\gamma_m = \inf \{ h(u,v) : u,v \in U_m, \, d_{cc}(u,v) \leq m^{-1} \}.
$$

Let $d_{cc}(u,v) = l$. By definition of $d_{cc}$ for an arbitrary $\varepsilon > 0$
there exists piecewise smooth path $\gamma : [0, l + \varepsilon] \to \mathcal{M}$
such that $\gamma(0) = u$, $\gamma(l + \varepsilon) = v$ and $|\dot\gamma| \leq 1$.
Let $w = \gamma(\frac{l+\varepsilon}{2})$.
Then $d_{cc}(u,w) \leq \frac{l+\varepsilon}{2}$ and
$d_{cc}(v,w) \leq \frac{l+\varepsilon}{2}$. Consequently,
$B_{cc}(w, \frac {l-\varepsilon}{2}) \subset B_{cc}(u,l)$ and
$B_{cc}(w, \frac {l-\varepsilon}{2}) \subset B_{cc}(v,l)$.
Hence,
$$
  h(u,v) \geq \frac
  {\mathcal{H}^\nu \bigl( B_{cc}(w, \frac {l-\varepsilon}{2}) \bigr)}
  {l^\nu}
  \geq \frac { C_1 (\frac {l-\varepsilon}{2})^\nu } {l^\nu} > 0,
$$
where $C_1 > 0$ is a constant from Ball--Box theorem.
Since $\varepsilon > 0$ is arbitrary, we infer
$\gamma_m \geq C_1 2^{-\nu} > 0$.

For every $m \in \mathbb{N}$ let $E^m$ be a set of all density points of
$E \cap (U_m \setminus U_{m-1} )$ (assuming $U_0 = \emptyset$).
The sequence of $E^m$ is a disjoint family and
$\mathcal{H}^\nu (E \setminus \bigcup\limits_{m=1}^{\infty} E_m ) = 0$.

For $k\in \mathbb N$,  $u \in E$, $0 < r < m^{-1}$ define
$$
    Q_k^m(u, r) = B_{cc}(u, r) \cap \{ x: x \not\in E^m \text{~or~}
        \widetilde d_{cc}( f(x), f(u) ) > k \, d_{cc}(x, u) \}
$$
and also define
$$
    B_k^m = E \cap \Bigl\{ u : \mathcal{H}^\nu( Q_k^m(u, r) ) < \gamma_m \frac{r^\nu}{2}
        \text{~for~all~} 0 < r < \min \{k^{-1}, m^{-1}\} \Bigr\}.
$$
By Property~\ref{Property1} all $B_k^m$ are measurable and
$E^m = \bigcup\limits_{k=1}^{\infty} B_k^m$.
Next, if $u,v \in B_k^m$ and $r = d_{cc}(u, v) < \min \{ k^{-1}, m^{-1} \}$ we have
$$
    \mathcal{H}^\nu( Q_k^m(u, r) \cup Q_k^m(v, r) ) < \gamma_m r^\nu
    	\leq \mathcal{H}^\nu( B_{cc}(u, r) \cap B_{cc}(v, r) ).
$$
Hence we can choose a point
$$
    x \in  (B_{cc}(u, r) \cap B_{cc}(v, r)) \setminus (Q_k^m(u, r) \cup Q_k^m(v, r)).
$$
For that point
\begin{align*}
    \widetilde d_{cc}(f(u),f(v))
    & \leq \widetilde d_{cc}(f(u),f(x)) + \widetilde d_{cc}(f(x),f(v)) \\
    & \leq k d_{cc}(u,x) + k d_{cc}(x,v)
    \leq 2 kr = 2 k d_{cc}(u, v).
\end {align*}
Consequently, representing $B_k^m$ as union of countable family
of measurable sets $B_{k,j}^m$, whose diameters are less than
$\min \{ k^{-1}, m^{-1} \}$, we see that every restriction
$f|_{B_{k,j}^m}$ is a Lipschitz mapping.
\end{proof}

\begin{proof}[Proof of Theorem~\ref{Theorem:ApprStepanoffLike}]
By Theorem~\ref{Theorem:SetFamilyLipschitz} the domain of $f$
is an union of countable family of disjoint sets $E_i$ such that
every $f|_{E_i}$ is a Lipschitz mapping (up to the set of
measure $0$). By Theorem~\ref{Theorem:RademacherLike} every
$f|_{E_i}$ is differentiable almost everywhere in $E_i$.
For the density points of $E_i$ this is equivalent to approximate
differentiability in $E$.
\end{proof}

\section{Theorem on approximate differentiability}

Now we have all necessary tools for formulating and proving the main result.
\begin{theorem} \label{Theorem:WorkResultDouble}
    Let $E \subset \mathcal{M}$ be a measurable subset of the
    Carnot--Carath\'eodory space $\mathcal{M}$ and let
    $f : E \to \widetilde{\mathcal{M}}$ be a measurable mapping.
    The following statements are equivalent:

\noindent $\quad 1)$ The mapping $f$ is approximately differentiable
        almost everywhere in $E$.

\noindent $\quad 2)$ The mapping $f$ has approximate derivatives
        $\ap X_j f$ along the basic
        horizontal vector fields $X_1,\dots,X_{\dim H_1}$
        almost everywhere in $E$.

\noindent $\quad 3)$ There is a sequence of disjoint sets
        $Q_1, Q_2, \dots$ such that 
        $\mathcal{H}^\nu(E \setminus \bigcup\limits_{i=1}^{\infty} Q_i) = 0$ and
        every restriction $f|_{Q_i}$ is a Lipschitz mapping.
\end{theorem}

\begin{proof}[Proof of the implication $1) \Rightarrow 3)$]
Let $g \in \mathcal{M}$ be a density point of $E$ and
Let $f$ be approximately differentiable in $g$.
Fix a point $v$ in a set
$$
	C_\varepsilon(g) = \{ v \in B_{cc}(g,r_g) \cap \mathcal{G}^g
	  : \widetilde d_{cc} (f(v), \ap D_g f(v)) < \varepsilon d_{cc}(g,v) \},
	\quad \varepsilon > 0.
$$
By Theorem~\ref{Th:LocApprForDcc} we have
\begin{align*}
	\widetilde d^{f(g)}_{cc} (f(v), \ap D_g f(v))
	& \leq \widetilde d_{cc} (f(v), \ap D_g f(v)) [1 + o(1)] \\
	& < d_{cc}(v,g) [\varepsilon + o(\varepsilon)]
	  = d^g_{cc}(v,g) [\varepsilon + o(\varepsilon)].
\end{align*}
From the definition of an approximate differential it follows that
$\mathcal{H}^\nu$-density of the set $B_{cc}(g,r_g) \setminus C_\varepsilon(g)$
equals zero for any $\varepsilon > 0$. In other words
$$
	\ap \lim_{v \to g}
	\frac {\widetilde d_{cc}(f(v), D_g f(v))} {d_{cc}(g,v)} = 0.
$$

Therefore,
\begin{align*}
    & \ap \varlimsup_{v \to g} \frac {\widetilde d_{cc}(f(g), f(v))} {d_{cc}(g,v)} \\
    & \quad \leq \ap \varlimsup_{v \to g} \frac {\widetilde d_{cc}(f(g), D_g f(v))} {d_{cc}(g,v)}
      + \ap \varlimsup_{v \to g} \frac {\widetilde d_{cc} (D_g f(v), f(v))} {d_{cc}(g,v)} \\
    & \quad = \varlimsup_{v \to g} \frac {\widetilde d_{cc}(f(g), D_g f(v))} {d^g_{cc}(g,v)} + 0 \\
    & \quad \leq \varlimsup_{v \to g} \frac {\widetilde d^{f(g)}_{cc}(f(g), D_g f(v))[1+o(1)]} {d^g_{cc}(g,v)} \\
    & \quad = [1 + o(1)] \sup_{v : \: d^g_{cc}(v,g) = 1} \widetilde d^{f(g)}_{cc}(f(g), D_g f(v)) < \infty
\end{align*}
for almost all $g \in E$.
Hence, the conditions of Theorem~\ref{Theorem:SetFamilyLipschitz} are fulfilled.
\end{proof}

The implication $3) \Rightarrow 2)$ is proved as
Corollary~\ref{Corollary:LipschitzDerives} in the next subsection.

The implication $2) \Rightarrow 1)$ is a direct corollary of the
following crucial
\begin{theorem} \label{Theorem:Main}
    Let $f : \mathcal{M} \to \widetilde{\mathcal{M}}$ be a measurable mapping of
    Carnot--Carath\'eodory spaces. Then
    \begin{align*}
        & A_j = \dom \ap X_j f \text{~is a measurable set}, \\
        & \ap X_j f : A_j \to \widetilde\exp(H \widetilde{\mathcal M})
            \text{~is a measurable mapping in~} A_j,
    \end{align*}
    for all $j = 1,\dots,\dim H_1$, and $f$ is approximately
    differentiable almost everywhere on the set
    $A = \bigcap\limits_{j=1}^{\dim H_1} A_j$. Moreover, if
    $g \in A$ is a point of an approximate differentiability of
    the mapping $f$ and in the neighborhood of $g$ we have representation
    from Theorem~$\ref{Theorem:CoordSys}$
    $$
        v = \exp (a_L X_{j_L}) \circ \dots \circ \exp (a_1 X_{j_1}) (g)
    $$
    where $1 \leq j_i \leq \dim H_1$, $i=1,\dots,L$,
    $L \in \mathbb{N}$, then
    $$
        \ap D_gf (v) = \prod_{i=1}^{L}
            \delta^{f(g)}_{a_i} \ap X_{j_i} f(g) \in \mathcal{G}^{f(g)}.
    $$
\end{theorem}
We follow the proof in \cite{Bib:Vod-Groups} where the similar result
was established for Carnot groups (which in turn was inspired by the proof \cite{Bib:Federer}
of the similar theorem for mappings of Euclidean spaces).
The essential steps of the proof are
carried out in separate lemmas which are proved below and the proof of the
theorem itself is located in the subsection \ref{Section:MainTheoremProof}
just after proofs of lemmas.

\subsection{Approximate derivatives}
\label{Section:ApproxDerivatives}
\begin{lemma} \label{Lemma:HorizDerivatives}
  Let $E \subset \mathcal{M}$ be a measurable set and
  $f : E \to \widetilde{\mathcal{M}}$
  be a measurable mapping. Then
  \begin{align*}
    & A_j = \{ x \in E : \ap \varlimsup_{t \to 0}
      \frac {\widetilde d_{cc}(\, f(x), \, f(\exp tX_j(x)) \, )} {|t|}
        < \infty \} \text{~is measurable}; \\
    & \ap X_j f : A_j \to \widetilde{\mathcal{M}}
      \text{~is defined almost everywhere and is measurable}; \\
      & \ap X_j f(g) \in \widetilde\exp(H_g \widetilde{\mathcal{M}})
        \text{~for almost all~} g \in A_j
  \end{align*}
  for every $j=1,\dots,\dim H_1$.
\end{lemma}

\begin{proof}
Fix $j \in \{1,\dots, \dim H_1 \}$. A mapping
$$
t \mapsto |t|^{-1} \widetilde d_{cc}(\, f(x), f(\exp t X_j(x))\,)
$$
is measurable
and by Property~\ref{Property2} the set $A_j$ is measurable.
For every $x \in E$ define $A_x$ as a set of real numbers
$t$ such that $\exp t X_j(x) \in A_j$.
In the case $A_x \not= \emptyset$ define also the mapping
$h : A_x \to \widetilde{\mathcal M}$ by the rule $h(t) = f( \exp t X_j(x) )$.

If $y = \exp t X_j(x)$, $t \in A_x$, we have
\begin{align*}
    & \ap \varlimsup_{\tau \to 0}
        \frac {\widetilde d_{cc}( h(t), h(t+\tau) )} {|\tau|} \\
    & \quad = \ap \varlimsup_{\tau \to 0}
        \frac {\widetilde d_{cc}( f(\exp tX_j(x)), f(\exp(t+\tau)X_j(x)) )}{|\tau|} \\
    & \quad = \ap \varlimsup_{\tau \to 0}
        \frac {\widetilde d_{cc}( f(\exp tX_j(x)), f(\exp \tau X_j( \exp t X_j(x))) )}{|\tau|} \\
    & \quad = \ap \varlimsup_{\tau \to 0}
        \frac {\widetilde d_{cc}( f(y), f(\exp \tau X_j(y)))} {|\tau|} < \infty.
\end{align*}
Hence, $h$ meets the conditions of Theorem~\ref{Theorem:SetFamilyLipschitz}.
Therefore, $A_x = B_0 \cup \bigcup\limits_{i=1}^{\infty} B_i$, where
$\mathcal{H}^\nu(B_0) = 0$, all $B_i$, $i=1,\dots,\infty$, are
measurable and restriction of $h$ on every $B_i$ is a Lipschitz mapping.
If $h : B_i \to \widetilde{\mathcal M}$ is one of these restrictions then by
Theorem~\ref{Theorem:DifLipschitz} the sub-Riemannian derivative
$$
    \Bigl. \frac{d}{d\tau}_{sub} h(t+\tau)
        \Bigr|_{\mathop{}^{\tau = 0}_{t+\tau \in B_i}}
        \in \widetilde\exp H_{h(t)} \widetilde{\mathcal M}
$$
exists for almost all $t$.
If $t$ is a density point for the set $B_i$ then
\begin{align*}
    \Bigl. \frac{d}{d\tau}_{sub} h(t+\tau) \Bigr|_{\mathop{}^{\tau = 0}_{t+\tau \in B_i}}
    & = \Bigl. \ap \frac{d}{d\tau}_{sub} h(t+\tau) \Bigr|_{\tau = 0} \\
    & = \Bigl. \ap \frac{d}{d\tau}_{sub} f(\exp{\tau X_j(y)}) \Bigr|_{\tau = 0}
    = \ap X_j f(y).
\end{align*}
Thus, $\ap X_j f(y)$ exists in $\{ y=\exp t X_j(x) : t \in A_x\}$ for
almost all $t \in A_x$. This provides existence of the derivative $\ap X_j f$
almost everywhere in $A_j$.
\end{proof}

\begin{corollary} \label{Corollary:LipschitzDerives}
    A Lipschitz mapping $f$ has approximate derivatives $\ap X_j f$
    along the horizontal vector fields $X_j$ almost everywhere and
    $\ap X_j f(g) \in \widetilde \exp( H_g \mathcal{M})$ for almost all
    $g \in \dom f$.
\end{corollary}

\begin{remark}
Note that if $\ap X_j f(g)$ defined at
$g \in \mathcal{M}$ then
$\ap (a X_j) f(g)$ is also defined for all real numbers $a$. Moreover
$$
    \ap (a X_j) f(g) = \tilde \delta^{f(g)}_a \ap X_j f(g).
$$
\end{remark}

Let the coordinate system \eqref{Eq:CoordPhiSystem} be defined in
a neighborhood of a point $g \in \mathcal{M}$. Consider a curve
\begin{equation} \label{Eq:GammaCurves}
  \Gamma_k(g; t) = \Phi_k(t)(g).
\end{equation}
We say that the mapping $f$ is \textit{approximately differentiable}
along the curve $\Gamma_k(g; t)$ at $t=0$ if there is an element
$a \in \mathcal{G}^{f(g)} \cap \widetilde{\mathcal M}$ such that
$$
	\frac 1{r^{\deg X_k}} \mathcal{H}^{\deg X_k}
	\Bigl\{ t \in (-r, r) :
	\frac { \widetilde d_{cc}^{f(g)} ( f \circ \Gamma_k(g; t), \tilde \delta^{f(g)}_t a ) }
	      { d_{cc}^{g} ( g, \Gamma_k(g;t)) }
	> \varepsilon
	\Bigr\} \to 0 \quad as~~r \to 0.
$$
We denote this derivative by $a = \ap d_{sub} (f\circ\Gamma_k)(g)$.
If $k = 1,\dots,\dim H_1$, this definition coincides with the definition
of the approximate derivative from Subsection~\ref{Section:ApproximateDiff}.

\begin{lemma} \label{Lemma:DerivativesAlongCurves}
    Let $E \subset \mathcal{M}$ be a bounded measurable set
    and $f : E \to \widetilde{\mathcal{M}}$
    be a measurable mapping. Let also the coordinate system
    \eqref{Eq:CoordPhiSystem} be defined at the neighborhood of a point
    $g \in U$ with functions $\Phi_k$ satisfying \eqref{Eq:CoordPhiLine}.
    Then the mapping $f$ is approximately differentiable along the curve
    $\Gamma_k(g; t)$ defined by \eqref{Eq:GammaCurves},
    $k=\dim H_1+1,\dots,N$, at $t = 0$ almost everywhere in
    $A = \bigcap\limits_{j=1}^{\dim H_1} \dom \ap X_j f$.
    Moreover, the approximate derivative can be written as
    \begin{align} \label{Eq:CurveDerivativeStructure}
        \ap d_{sub} (f\circ\Gamma_k)(g) 
        & = \ap (s_{L_k} \widehat X^g_{j_{L_k}}) f \circ \dots
            \circ \ap (s_1 \widehat X^g_{j_1}) f(g) \notag \\
        & = \ap (s_1 \widehat X^g_{j_1}) f(g) \cdot \ldots
            \cdot \ap (s_{L_k} \widehat X^g_{j_{L_k}}) f(g) \in \mathcal{G}^{f(g)},
    \end{align}
    almost everywhere. Here $L_k \leq L$ and $s_i = \pm1$ are from the representation
    \eqref{Eq:CoordPhiLine}. Also the following estimate
    \begin{multline} \label{Eq:CurveDerivativeEstimate}
        \widetilde d_{cc}^{f(g)} \bigl( f(g), \ap d_{sub} (f\circ\Gamma_k)(g) \bigr) \\
        \leq L_k
            \max \{ \widetilde d_{cc} \bigl({f(g)}, \ap X_j f(g) \bigr) : j = 1,\dots,\dim H_1 \}
    \end{multline}
    holds for all $k = \dim H_1+1, \dots, N$.
\end{lemma}

A sketch of the proof:

At the first step we apply Luzin's and Egorov's theorems
to a bounded set $A$ and obtain a set $A' \subset A$ that differs
from $A$ on a set of a measure small enough and on which the limit
$\ap \lim\limits_{t \to 0} \widetilde \delta^{f(g)}_{t^{-1}} f (\exp t X_j)(g)$
converges to $\ap X_j f(g)$ uniformly.

Next we assure that the set of real numbers $t$, for which the relation
\eqref{Eq:CurveDerivativeStructure} does not hold, is negligible.

At last, we prove that the uniform limit
$\ap \lim\limits_{t \to 0} \widetilde\delta^{f(g)}_{t^{-1}} f \circ \Gamma_k(g;t)$
converges to the
\eqref{Eq:CurveDerivativeStructure} in $A'$.

\begin{proof}
By Lemma~\ref{Lemma:HorizDerivatives} the sets $A_j = \dom \ap X_j f \subset E$
are measurable and the mappings $\ap X_j f$ are measurable in $A_j$ for
all $j = 1,\dots,\dim H_1$.

We have $\mathcal{H}^\nu(A_j) \leq \mathcal{H}^\nu(E) < \infty$. Fix $\varepsilon > 0$.
Applying Luzin's theorem we find a closed set
$E' \subset A$ such that
$\mathcal{H}^\nu( A \setminus E' ) < \varepsilon/2$ and all $\ap X_j f$ are
uniformly continuous in $E'$.

Consider a sequence of functions
$\{ \varphi^j_n : E' \to \mathbb{R} \}_{n \in \mathbb{N}}$
defined as
$$
	\varphi^j_n(g) = \sup_{|t| < \frac 1n}
	\frac {\widetilde d_{cc}^{f(g)} \bigl( f(\exp(t X_j)(g)),
	\widetilde \delta^{f(g)}_t \ap X_j f(g) \bigr)} {|t|},
	\quad j=1,\dots,\dim H_1.
$$
Since $\varphi^j_n(g) \mathop{\to}\limits_{ap} 0$ as $n \to \infty$,
by Egorov's theorem we obtain a measurable set $E'' \subset E'$ such that 
$\mathcal{H}^\nu( E' \setminus E'' ) < \varepsilon/2$ and
$\varphi^j_n(g) \to 0$ as $n \to \infty$ uniformly on $E''$.
Therefore, the limits
$$
	\ap \lim_{t \to 0} \frac
		{\widetilde d_{cc}^{f(g)} \bigl( f(\exp(t X_j)(g)),
		\widetilde \delta^{f(g)}_t \ap X_j f(g) \bigr)}
		{|t|} = 0 
$$
converge uniformly on $E''$ for all $j = 1,\dots,\dim H_1$.

For every positive integer $m$ and for all $x \in E$, $r > 0$ define a set
$$
    T^m_j(x,r) = \Bigl\{
        t \in (-r, r) :
        \widetilde d_{cc}^{f(x)} \bigl( f(\exp tX_j(x)),
        \tilde\delta^{f(g)}_t \ap X_j f(x) \bigr)
        > \frac{|t|}{m}
    \Bigr\}.
$$
For all positive integers $p$ we introduce
$$
    B^m_j(p) = A_j \cap \Bigl\{ x \in E  :
        \mathcal{H}^1[ T^m_j(x,r) ] \leq \frac{r}{m} \text{~for all~} 0 < r < p^{-1}
    \Bigr\}.
$$
By Property~\ref{Property1} the sets $B^m_j(p)$ are measurable for all
$j = 1,\dots, \dim H_1$. We have also $\bigcup\limits_{p=1}^{\infty} B^m_j(p) = A_j$.

Moreover, $B^m_j(p) \subset B^m_j(p+1)$.
Hence, we can choose a sequence of numbers $p_1, p_2, \dots$ such that
$\mathcal{H}^\nu(E'' \setminus B^m_j(p_m)) < \frac{\varepsilon}{2^m}$ holds.
Therefore,
$$
    \mathcal{H}^\nu(E'' \setminus F) < \varepsilon\!\cdot\!\dim H_1, \quad
        \text{~where~} F = \bigcap_{j=1}^{\dim H_1} \bigcap_{m=1}^{\infty} B^m_j(p_m).
$$
Next, for all $x \in F$, $r > 0$ define a set
$$
    Z_j(x,r) = \{ y = \exp t X_j(x) : |t|<r \text{~and~} y \not\in F \}, \quad j=1,\dots,\dim H_1.
$$
For all positive integers $m$ and $q$ define the sets
$$
    C^m_j(q) = F \cap \Bigl\{ x \in E :
        \mathcal{H}^1[Z_j(x,r)] \leq \frac{r}{2^m} \text{~for all~} 0 < r < q^{-1}
    \Bigr\}.
$$
By Property~\ref{Property1} all $C^m_j(q)$ are measurable. Also
$\mathcal{H}^\nu \Bigl( F \setminus \bigcup\limits_{q=1}^{\infty} C^m_j(q) \Bigl) = 0$.

Moreover, $C^m_j(q) \subset C^m_j(q+1)$. Hence, we can choose a sequence
of numbers $q_1, q_2, \dots$ such that
$\mathcal{H}^\nu(F \setminus C^m_j(q_m)) < \frac{\varepsilon}{2^m}$ holds.
Therefore,
$$
    \mathcal{H}^\nu(F \setminus F_1) < m\varepsilon, \quad
        \text{~where~} F_1 = \bigcap_{j=1}^{\dim H_1} \bigcap_{n=1}^{\infty} C^n_j(q_n).
$$

Next, we prove that the function $f$ is approximately differentiable
along the curve $\Gamma_k(g;t)$ uniformly in $F_1$ and the mapping
$g \mapsto \ap \bigl. \frac{d}{dt}_{sub} f(\Gamma_k(g;t)) \bigr\vert_{t=0}$
is uniformly continuous in $F_1$.

Fix $m \in \mathbb{N}$, $0 < r < \min \{ p_m^{-1}, q_m^{-1} \}$ and a
density point $g \in F_1$. Denote
\begin{align*}
    u_1(t) & = \exp (t s_1 X_{j_1})(g), \\
    u_i(t) & = \exp (t s_i X_{j_i})( u_{i-1}(t) ), \quad i=2,\dots,L_k.
\end{align*}
Then $u_{L_k}(t)=\Gamma_k(g;t)$. Define the set $S^m \subset (-r, r)$
as follows:
\begin{align*}
    t \in S^m, \text{~if~} & s_1 t \in T^m_{j_1}(g, r), \\
         \text{~or~} & s_i t \in T^m_{j_i}( u_{i-1}(t) , r), \\
         \text{~or~} & u_1(t) \in Z_{j_1}(g, r), \\
         \text{~or~} & u_i(t) \in Z_{j_i}( u_{i-1}(t) , r), \quad i=2,\dots,L_k. \\
\end{align*}
Since $\mathcal{H}^1[T^m_{j_1}(g, r)] \leq \frac{r}{m}$,
$\mathcal{H}^1[Z_{j_1}(g,r)] \leq \frac{r}{m}$ and since
$\mathcal{H}^1[T^m_{j_i}( u_{i-1}(t) , r)] \leq \frac{r}{m}$,
$\mathcal{H}^1[Z_{j_i}( u_{i-1}(t) , r)] \leq \frac{r}{m}$ if
$u_{i-1}(t) \in F_1$, $i=2,\dots,L_k$, we have
$$
    \mathcal{H}^1(S^m) \leq 2 L_k \frac{r}{m}.
$$

Now we estimate $\mathcal{H}^{\deg X_k}$-measure of the set $S^m$.
Fix arbitrary numbers
\begin{equation} \label{Eq:DeltaUndLambda}
	\delta > 0 \text{~and~} \Lambda > 2 L_k \frac{r}{m}.
\end{equation}
Cover the set $S^m$ with a countable family of intervals $(a_\xi, b_\xi)$
so that
$$
    b_\xi - a_\xi < \delta, \quad \sum_\xi (b_\xi - a_\xi) < \Lambda.
$$
Then
$$
    |b_\xi-a_\xi|^{\deg X_k} < \delta (2r)^{\deg X_k-1}, \quad
        \sum_\xi |b_\xi-a_\xi|^{\deg X_k} < \Lambda (2r)^{\deg X_k-1}.
$$
Since $\delta$ and $\Lambda$ are arbitrary numbers of
\eqref{Eq:DeltaUndLambda}, we have
$$
    \mathcal{H}^{\deg X_k}(S^m)
        \leq 2^{\deg X_k} L_k \frac{r^{\deg X_k}}{m}.
$$

Now we show that the expression \eqref{Eq:CurveDerivativeStructure}
is really the derivative of the composition $f \circ \Gamma_k$.
For the points $u, v \in F_1$ we have
\begin{align*}
    & \widetilde d_{cc}^{f(g)} \bigl( \, f(\exp(t s_i X_i)(v)), \,
        \ap (t s_i X_i)f(v) \, \bigr) \leq \varphi_i(t), \\
    & \widetilde d_{cc}^{f(g)} \bigl( \, \ap(t s_i X_i )f(u), \,
        \ap(t s_i X_i)f(v) \, \bigr) \leq t \omega_i( d_{cc}(u,v) ),
\end{align*}
where $\frac{\varphi_i(t)}{t} \to 0$ as $t \to 0$ uniformly for $v \in F_1$ and
$\omega_i(t)$ are moduli of continuity of the mappings $\ap (s_i X_i)f(\cdot)$
in $F_1$, $i = 1, \dots, \dim H_1$.

If $|t|<r$ and $t \not\in S^m$ we obtain
\begin{align*}
    & \widetilde d_{cc}^{f(g)} \bigl( f \circ u_1(t),\,
        \tilde\delta^{f(g)}_{t} \ap(s_1 X_1)f(g) \bigr) \\
    & \quad = \widetilde d_{cc}^{f(g)} \bigl( f \circ \exp(t s_1 X_1)(g),\,
         \ap(t s_1 X_1)f(g) \bigr) \\
    & \quad \leq \varphi_1(t) = C_1(t).
\end{align*}
Further, by induction:
\begin{align*}
    & \widetilde d_{cc}^{f(g)} \Bigl( f \circ u_j(t) ,\,
        \tilde\delta^{f(g)}_{t} \prod_{i=1}^{j} \ap(s_i X_i)f(g) \Bigr) \\
    & \ = \widetilde d_{cc}^{f(g)} \Bigl( f \circ \exp(t s_j X_j) (u_{j-1}(t)) ,\,
        \ap(t s_j X_j)f \circ \prod_{i=1}^{j-1} \ap(t s_i X_i)f(g) \Bigr) \\
    & \ \leq  \widetilde d_{cc}^{f(g)} \bigl( f \circ \exp(t s_j X_j) (u_{j-1}(t)) ,\,
        \ap(t s_j X_j)f(u_{j-1}(t)) \bigr) \\
    & \qquad + \widetilde d_{cc}^{f(g)} \Bigl( \ap(t s_j X_j)f(u_{j-1}(t)) ,\,
        \ap(t s_j X_j)f \circ \prod_{i=1}^{j-1} \ap(t s_i X_i)f(g) \Bigr) \\
    & \ \leq \varphi_j(t) + t \omega_j \bigl( C_{j-1}(t) \bigr)
        = C_j(t),
\end{align*}
where $\frac {C_j(t)}{t} \to 0$ as $t \to 0$ uniformly for $g \in F_1$.

Therefore, for $t \in (-r,r) \setminus S^m$ we have an evaluation
$$
    \widetilde d_{cc}^{f(g)} \Bigl( f(\Gamma_k(g;t)) ,\,
        \tilde\delta^{f(g)}_{t} \prod_{i=1}^{L_k} \ap(s_i X_i)f(g) \Bigr) = o(t),
$$
i.~e. the equality
$$
    \ap \Bigl. \frac{d}{dt}_{sub} f(\Gamma_k(g;t)) =
        \prod_{i=1}^{L_k} \ap(s_i X_i)f(g)
$$
holds for $g \in F_1$. Since $r$, $m$, $\varepsilon$ are arbitrary the latter takes place almost
everywhere in $E$. The inequality \eqref{Eq:CurveDerivativeEstimate} follows from
\eqref{Eq:CurveDerivativeStructure} and the generalized triangle inequality.
\end{proof}

\begin{remark}
Consider in the previous lemma the curve $\Gamma'_k(g; t) = \Gamma_k(g; \lambda t)$,
$\lambda \in \mathbb{R}\setminus\{0\}$, instead of $\Gamma_k(g; t)$.
The following representation takes place
$$
    \Gamma'_k(g; t) = \exp (\lambda t s_{L_k} X_{j_{L_k}}) \circ \dots
        \circ \exp(\lambda t s_1 X_{j_1})(g),
$$
where $s_i = \pm 1$, $1 \leq j_i \leq \dim H_1$. Then if there is    
$\ap d_{sub} (f\circ\Gamma_k)(g)$ defined at the point $g \in \mathcal{M}$
the derivative $\ap d_{sub} (f\circ\Gamma'_k)(g)$ is also defined
and we have
\begin{align} 
    \notag
         \ap d_{sub} (f\circ\Gamma'_k)(g) &
            = \prod_{i=1}^{L_k} \ap (\lambda s_i X_i) f(g)
    \\ \label{Eq:GammaDilation}
        &= \tilde \delta^{f(g)}_\lambda \prod_{i=1}^{L_k} \ap (s_i X_i) f(g)
        = \tilde \delta^{f(g)}_\lambda \ap d_{sub} (f\circ\Gamma_k)(g).
\end{align}
\end{remark}

\subsection{Construction and properties of a mapping of local groups}
\label{Section:LMapping}
Consider the system of the coordinates of the second kind \eqref{Eq:PhiHatSystem}
in a neighborhood $V(g) \subset \mathcal{G}^g$ of $g$. Define a mapping
$L_g : V(g) \to \mathcal{G}^{f(g)}$ as follows:
\begin{equation} \label{Eq:LMapping}
    L_g : \hat v  = \widehat \Phi_g(t_1,\dots,t_N) \mapsto
        \prod_{k=1}^{N} \tilde\delta^{f(g)}_{t_k}
            \ap d_{sub} (f\circ\Gamma_k)(g).
\end{equation}

Declare some properties of this mapping.

\begin{property} \label{PropertyL0}
The mapping $L_g$ is continuous.
\end{property}

It follows directly from \eqref{Eq:LMapping}.

\begin{property} \label{PropertyL1}
$\tilde\delta^{f(g)}_\lambda \circ L_g = L_g \circ \delta^g_\lambda$.
\end{property}

Really, for $\hat v = \widehat \Phi_g(t_1,\dots,t_N)$ we have
\begin{align*}
    \delta^g_\lambda \hat v & =
        \delta^g_\lambda \widehat \Phi_g(t_1,\dots,t_N) \\
    & = \delta^g_\lambda \widehat \Phi_N(t_N) \circ \dots
        \circ \widehat \Phi_{\dim H_1 + 1}(t_{\dim H_1 + 1}) \\
    & \qquad \circ \exp(t_{\dim H_1} \widehat X^g_{\dim H_1}) \circ \dots
        \circ \exp(t_1 \widehat X^g_1)(g) \\
    & = \widehat \Phi_N(\lambda t_N) \circ \dots
        \circ \widehat \Phi_{\dim H_1 + 1}(\lambda t_{\dim H_1 + 1}) \\
    & \qquad \circ \exp(\lambda t_{\dim H_1} \widehat X^g_{\dim H_1}) \circ \dots
        \circ \exp(\lambda t_1 \widehat X^g_1)(g) \\
    & = \widehat \Phi_g(\lambda t_1, \dots, \lambda t_N).
\end{align*}
Then, taking into account \eqref{Eq:GammaDilation}, we get
\begin{align*}
  L_g(\delta^g_\lambda \hat v)
  & = \prod_{k=1}^{N} \tilde\delta^{f(g)}_{\lambda t_k}
    \ap d_{sub} (f\circ\Gamma_k)(g) \\
  & = \tilde\delta^{f(g)}_\lambda \prod_{k=1}^{N}
    \tilde\delta^{f(g)}_{t_k}
    \ap d_{sub} (f\circ\Gamma_k)(g)
    = \tilde\delta^{f(g)}_\lambda L_g(\hat v).
\end{align*}

\begin{property} \label{PropertyL2}
The mapping $L_g$ is bounded.
\end{property}

By Property~\ref{PropertyL1} the mapping $L_g$ is homogeneous, so
$$
    \Vert L_g \Vert
    = \sup_{v \not= g} \frac { \widetilde d_{cc}^{f(g)}( L_g(g), L_g(v)) } { d_{cc}^g(g,v) }
    = \sup_{d_{cc}^g(g,v)=1} \widetilde d_{cc}^{f(g)}(L_g(g), L_g(v)).
$$
The latter is finite because of continuity of $L_g$.

\begin{property} \label{PropertyL3}
Let $u$, $v \in \mathcal{G}^g$ be such that $d_{cc}^g(u,v) = o(d_{cc}^g(g,u))$ as
$u \to g$. Then
$$
	\widetilde d_{cc}^{f(g)}(L_g(u),L_g(v)) = o(d_{cc}^g(g,u)).
$$
\end{property}

Let $\omega(t)$ be a modulus of continuity of the mapping
$L_g : B_{cc}(g,2) \to \mathcal{G}^{f(g)}$. Then if we define
$r = \max \{ d_{cc}^g(g,u), d_{cc}^g(g,v) \}$
by Property~\ref{PropertyL1} we have
\begin{multline*}
    \widetilde d_{cc}^{f(g)}( L_g(u), L_g(v) ) =
    O(r)\, \widetilde d_{cc}^{f(g)}( L_g(\delta^g_{r^{-1}} u), L_g(\delta^g_{r^{-1}} v) ) \\ \leq
    O(r)\, \omega \Bigl( \frac {d_{cc}^g(u,v)} {r} \Bigr) = r\, o(1)
    \quad \text{~as~} r \to 0.
\end{multline*}

\begin{lemma} \label{Lemma:CoordinateDifferentiability}
    Let $E \subset \mathcal{M}$ be a bounded measurable set
    and let $f : E \to \widetilde{\mathcal M}$ be a measurable mapping.
    Let the coordinate system of the 2nd kind
    \eqref{Eq:CoordPhiSystem} be defined in a neighborhood of $g \in \mathcal{M}$.
    Then the mapping $f$ is approximately differentiable along the curves
    $\Gamma_k(g;t)$, $k=1,\dots,N$, almost everywhere in
    $A = \bigcap\limits_{j=1}^{\dim H_1} \dom \ap X_j f$ and the equality
    \begin{equation} \label{Eq:LDiffers}
        \ap \lim_{v \to g} \frac {\widetilde d_{cc}^{f(g)}(f(v), L_g(v))} {d_{cc}^g(g,v)} = 0
    \end{equation}
    holds for almost all $g \in A$, where $L_g$ is the mapping
    defined by the formula \eqref{Eq:LMapping}.
\end{lemma}

\begin{proof}
By Lemma~\ref{Lemma:HorizDerivatives} all sets $A_j = \dom \ap X_j f$
are measurable and by Lemma~\ref{Lemma:DerivativesAlongCurves}
$f$ is approximately differentiable along the curves $\Gamma_k$,
$k=1,\dots,N$, almost everywhere in $A$.

Fix $\varepsilon > 0$. By Luzin's theorem there is a measurable
set $E' \subset A$ such that $\mathcal{H}^\nu(A \setminus E') < \varepsilon / 2$
and the mappings $E' \owns x \mapsto \ap d_{sub} (f\circ\Gamma_k)(x)$
are uniformly continuous for all $k=1,\dots,N$.

Consider a sequence of functions
$\{ \varphi^k_n : E' \to \mathbb{R} \}_{n \in \mathbb{N}}$
defined as
$$
	\varphi^k_n(g) = \sup_{|t| < \frac 1n}
	\frac {\widetilde d_{cc}^{f(g)} \bigl( f(\Gamma_k(v;t)) ,\,
        \tilde\delta^{f(v)}_t \ap d_{sub} (f\circ\Gamma_k)(v) \bigr)} {|t|},
   \quad k = 1,\dots,N.
$$
We have $\varphi^k_n(g) \mathop{\to}\limits_{ap} 0$ as $n \to \infty$ in $E'$.
By Egorov's theorem there is $E'' \subset E'$ such that
$\mathcal{H}^\nu(E' \setminus E'') < \varepsilon/2$ and
$\varphi^k_n(g) \to 0$ as $n \to \infty$ uniformly on $E''$.

For every positive integer $m$ and for all $x \in E$, $r > 0$ we define the set
$$
	T^m_k(x,r) = \Bigl\{ t \in (-r,r) :
    \widetilde d_{cc}^{f(x)} \bigl( f(\Gamma_k(x;t)),
        \tilde\delta^{f(x)}_t \ap d_{sub} (f\circ\Gamma_k)(x) \bigr) > \frac{|t|}{m}
    \Bigr\}.
$$
For all positive integers $p$, we define the set
$$
    B^m_k(p) = A \cap \Bigl\{ x \in E : \mathcal{H}^{\deg X_k}[T^m_k(x,r)]
        \leq \frac{r^{\deg X_k}}{m} \text{~for all~} r \in (0, p^{-1}) \Bigr\}.
$$
In the case $k > 1$ we also define $Z^m_k(x,r;p)$, as the set of the points
$z = (z_1,\dots,z_{k-1},0,\dots,0) \in \mathbb{R}^N$
such that $z \in B(0,r)$ and $\Phi_x(z) \notin B^m_k(p)$.
Finally, for every positive integer $q$, we define the set
$$
    C^m_k(p,q) = B^m_k \cap \Bigl\{ x \in E : \mathcal{H}^{h_{k-1}}[Z^m_k(x,r;p)]
        \leq \frac {r^{h_{k-1}}} {m} \text{~for all~} r \in (0, q^{-1}) \Bigr\}
$$
where $h_k = \sum\limits_{i=1}^k \deg X_i$.

By Property~\ref{Property1}, the sets $B^m_k(p)$, $C^m_k(p,q)$ are measurable and
\begin{align*}
    & A = \bigcup_{p=1}^{\infty} B^m_k(p), \\
    & \mathcal{H}^\nu \Bigl( B^m_k(p) \setminus \bigcup_{q=1}^{\infty} C^m_k(p,q) \Bigr) = 0
    \quad \text{~for all~} k=1,\dots,N; m \in \mathbb{N}.
\end{align*}
Moreover, $B^m_k(p) \subset B^m_k(p+1)$, $C^m_k(p,q) \subset C^m_k(p,q+1)$.
Hence, we can choose sequences of numbers
$p_1,p_2,\dots$ and $q_1,q_2,\dots$ such that
\begin{align*}
    & \mathcal{H}^\nu(E'' \setminus B^m_k(p_m)) < \frac{\varepsilon}{2^m}, \\
    & \mathcal{H}^\nu(E'' \cap B^m_k(p_m) \setminus C^m_k(p_m,q_m) ) < \frac{\varepsilon}{2^m}
\end{align*}
for all $k = 1,\dots,N$ and for every $m$. Then
$$
    \mathcal{H}^\nu(E'' \setminus F) < 2 N \varepsilon
        \quad \text{~where~} F = \bigcap_{k=1}^{N} \bigcap_{m=1}^{\infty} C^m_k(p_m,q_m).
$$

Next we show that the limit \eqref{Eq:LDiffers} converges uniformly in $F$. 
Really, we have uniform estimates:
\begin{align*}
    & \widetilde d_{cc}^{f(g)} \bigl( f(\Gamma_k(v;t)) ,\,
        \tilde\delta^{f(v)}_t \ap d_{sub} (f\circ\Gamma_k)(v) \bigr)
            \leq \varphi_k(t), \\
    & \widetilde d_{cc}^{f(g)} \bigl( \tilde\delta^{f(u)}_t \ap d_{sub} (f\circ\Gamma_k)(u) ,\,
        \tilde\delta^{f(v)}_t \ap d_{sub} (f\circ\Gamma_k)(v) \bigr)
            \leq t \omega_k( d_{cc}(u,v) )
\end{align*}
for all $k = 1,\dots,N$, $u,v \in F$, where $\frac{\varphi(t)}{t} \to 0$ as
$t \to 0$ uniformly for $v \in F$, $\omega_k(\cdot)$ are moduli of the continuity
of the mappings $\ap d_{sub} (f\circ\Gamma_k)$.

Fix a density point $g \in F$, $m \in \mathbb{N}$ and
$0 < r < \min \{ p_m^{-1}, q_m^{-1} \}$.
For every $k = 1,\dots,N$ define
$S_k \subset \mathbb{R}^N$ as the set of the points
$(t_1,\dots, t_N) \in B(0,r)$ such that
\begin{align*}
    & \text{~either~} k > 1 \text{~and~} (t_1, \dots, t_{k-1}) \in Z^m_k(g,r; p_m), \\
    & \text{~or~} t_k \in T^m_k ( \Phi_g(t_1, \dots, t_{k-1}, 0,\dots,0), r ).
\end{align*}
Since $\mathcal{H}^{h_{k-1}}[Z^m_k(g,r;p_m)] \leq \frac{r^{h_{k-1}}}{m}$
and since
$\mathcal{H}^{\deg X_k}[T^m_k(x,r)] \leq \frac{r^{\deg X_k}}{m}$ if
$x = \Phi_g(t_1,\dots, t_{k-1}, 0, \dots, 0) \in B^m_k(p_m)$, we have
$$
    \mathcal{H}^\nu(S_k) \leq C_1 \frac{r^{h_{k-1}}}{m} r^{\nu-h_{k-1}} +
        C_2 \frac{r^{\deg X_k}}{m} r^{\nu-\deg X_k} \leq C_3 \frac{r^\nu}{m}.
$$
If we use the notation $W = \bigcup\limits_{k=1}^{N} S_k$ then
$\mathcal{H}^\nu(W) \leq C_4 \frac{r^\nu}{m}$.
Denote
\begin{align*}
    u_1 & = \Gamma_1(g; t_1), \\
    u_k & = \Gamma_k(u_{k-1}; t_k)
    \quad \text{for all} \ k = 2,\dots,N.
\end{align*}
Now, if $v \in F \setminus W$ and $u_N(t) \in F \setminus W$, we have
$$
  \widetilde d_{cc}^{f(g)} \bigl( f(\Gamma_1(g; t_1)) ,\,
  \tilde\delta^{f(g)}_{t_1} \ap d_{sub} (f\circ\Gamma_1)(g) \bigr)
  \leq \varphi_1(t_1) = C_1(|t_1|),
$$
and then, by induction,
\begin{align*}
  & \widetilde d_{cc}^{f(g)} \Bigl( f(u_k) ,\, \prod_{l=1}^{k}
    \tilde\delta^{f(g)}_{t_l} \ap d_{sub} (f\circ\Gamma_l)(g) \Bigr) \\
  & \ \leq \widetilde d_{cc}^{f(g)} \bigl( f(\Gamma_k(u_{k-1};t_k)) ,\,
    \tilde\delta^{f(u_{k-1})}_{t_k}
    \ap d_{sub} (f\circ\Gamma_k)(u_{k-1}) \bigr) \\
  & \quad + \widetilde d_{cc}^{f(g)} \Bigl(
    \tilde\delta^{f(u_{k-1})}_{t_k}
    \ap d_{sub} (f\circ\Gamma_k)(u_{k-1}) ,\, \prod_{l=1}^{k}
    \tilde\delta^{f(g)}_{t_l} \ap d_{sub} (f\circ\Gamma_l)(g) \Bigr) \\
  & \ \leq \varphi_k(t_k) +
    |t_k| \omega_k(C_{k-1}(|t_1|+\dots+|t_{k-1}|))
    = C_k(|t_1|+\dots+|t_k|),
\end{align*}
where
$\max \{ |t_1|, \dots, |t_k| \}^{-1} C_k(|t_1|+\dots+|t_k|) \to 0$
as $t \to 0$ uniformly for $g \in F$.

Denoting $\hat v = \widehat \Phi_g(t_1,\dots,t_N)$ we finally obtain
$$
    \ap \lim_{v \to g} \frac {\widetilde d_{cc}^{f(g)}(f(v), L_g(\hat v))}
    {d_{cc}^g(g,v)} = 0.
$$

If $v = \Phi_g (t_1, \dots, t_N) \in F \cap \mathcal{G}^g$ then
$d_{cc}^g(v,\hat v) = o(d_{cc}^g(g,v))$ as $v \to g$ by
Theorem~\ref{Th:PathsEstimate}.
Hence, using Property~\ref{PropertyL3} of the mapping $L_g$ we have
$$
    \widetilde d_{cc}^{f(g)}( f(v), L_g(v))
    \leq \widetilde d_{cc}^{f(g)}(f(v), L_g(\hat v)) + \widetilde d_{cc}^{f(g)}(L_g(v), L_g(\hat v))
    = o( d_{cc}^g(g,v) )
$$
as $v \to g$.
Since $r$, $m$, $\varepsilon$ are arbitrary we have
$$
    \ap \lim_{v \to g} \frac {\widetilde d_{cc}^{f(g)}(f(v),L_g(v))} {d_{cc}^g(g,v)} = 0
$$
for almost all $g \in A$.
\end{proof}

\subsection{Proof of theorem on approximate differentiability}
\label{Section:MainTheoremProof}

\begin{lemma} \label{Lemma:ApprDiffForLipschitz}
	Let $E \subset \mathcal{M}$ be a measurable set,
	$f: E \to \mathcal{M}$ be a measurable
	mapping, $g$ be a density point of $E$ and let
	\begin{equation} \label{lapdiff}
 		\ap \lim_{v \to g} \frac{\widetilde d_{cc}^{f(g)}(f(v),L_g(v))}{d_{cc}^g(g,v)} = 0,
	\end{equation}
	where $L_g : \mathcal{G}^g \cap \mathcal{M} \to \mathcal{G}^{f(g)}$ enjoys
	Properties~\ref{PropertyL0} -- \ref{PropertyL3}.
	If there are $\eta > 0$,
	$0 < K < \infty$ such that 
	$$
		\widetilde d_{cc}(f(u),f(v)) < K d_{cc}(u,v)
 	$$
 	for all $u, v \in B(g,\eta)$, then there exists the uniform limit
 	\begin{equation} \label{ldiff}
 		\lim_{v \to g} \frac{\widetilde d_{cc}^{f(g)}(f(v),L_g(v))}{d_{cc}^g(g,v)} = 0.
 	\end{equation}
\end{lemma}

\begin{proof}
Let $\omega(t)$ be a modulus of continuity of
$L_g : B(g,2) \cap \mathcal{G}^g \to \mathcal{G}^{f(g)}$. Then if
$d_{cc}^g(u,v) < d_{cc}^g(g,v) < \eta$,
by Property~\ref{PropertyL1}, we have
\begin{multline*}
	\widetilde d_{cc}^{f(g)} \bigl( L(u), L(v) \bigr)
	= d_{cc}^g(g,v) \, \widetilde d_{cc}^{f(g)} \bigl( L(\delta^g_{d_{cc}^g(g,v)^{-1}} u),
		L(\delta^g_{d_{cc}^g(g,v)^{-1}} v) \bigr) \\
	\leq d_{cc}^g(g,v) \, \omega \bigl(d_{cc}^g(\delta^g_{d_{cc}^g(g,v)^{-1}} u,
		\delta^g_{d_{cc}^g(g,v)^{-1}} v) \bigr)
	= d_{cc}^g(g,v) \, \omega \Bigl( \frac{d_{cc}^g(u,v)}{d_{cc}^g(g,v)} \Bigr).
\end{multline*}
Suppose $0 < \varepsilon < 1$. Fulfillment of the condition \eqref{lapdiff}
means there exists $\delta > 0$ such that for any $0 < r < \delta$ and for the set
$$
	W = \{ z \in E: \widetilde d_{cc}^{f(g)}(f(z), L_g(z)) < \varepsilon d_{cc}^g(g,z) \}
$$
we have $\mathcal{H}^\nu( B(g,r) \setminus W) < r^\nu \varepsilon^\nu$.
If we take $x \in B(g, \delta(1-\varepsilon)) \cap E$ and 
$r = d_{cc}^g(g,x) / (1 - \varepsilon)$ then
$B(x, r\varepsilon) \subset B(g, r)$. It follows
$B(x, r\varepsilon) \cap W \not= \emptyset$, hence, we can choose
$z \in B(x, r\varepsilon) \cap E$. By Theorem~\ref{Theorem:MetricsDeviation}
we have $d_{cc}(x,z) = o(d_{cc}(g,x)) = o(d_{cc}^g(g,x))$ and
\begin{multline*}
	\widetilde d_{cc}^{f(g)}(f(x),f(z))
	= \widetilde d_{cc}(f(x),f(z)) + o \bigl( \widetilde d_{cc}(f(g),f(x)) \bigr) \\
	= \widetilde d_{cc}(f(x),f(z)) + o( d_{cc}^g(g,x)),
\end{multline*}
where all $o(\cdot)$ are uniform. Thus, we infer
\begin{align*}
	& \widetilde d_{cc}^{f(g)} (L_g(x), f(x)) \\
	& \quad \leq \widetilde d_{cc}^{f(g)}( L_g(x), L_g(z) )
	             + \widetilde d_{cc}^{f(g)}( L_g(z), f(z) )
	             + \widetilde d_{cc}^{f(g)}( f(z), f(x) ) \\
	& \quad \leq d_{cc}^g(g,x)\, \omega \Bigl( \frac{d_{cc}^g(x,z)}{d_{cc}^g(g,x)} \Bigr)
	             + \varepsilon \, d_{cc}^g(g,z)
	             + \widetilde d_{cc}(f(x),f(z))
	             + o(d_{cc}^g(g,x)) \\
	& \quad \leq d_{cc}^g(g,x)\, \omega (1) + \varepsilon \, d_{cc}^g(g,x)
	             + \varepsilon \, d_{cc}^g(x,z)
	             + K d_{cc}(x,z)
	             + o(d_{cc}^g(g,x)) \\
	& \quad \leq d_{cc}^g(g,x) \bigl( \omega (1) + \varepsilon
	             + (\varepsilon+K+1) o(1) \bigr)
\end{align*}
where all $o(\cdot)$ are uniform.
\end{proof}

\begin{remark}
	If we prove that the mapping $L_g$ is the approximate differential of $f$
	then from Lemma~\ref{Lemma:ApprDiffForLipschitz} it follows that the
	Lipschitz mapping is differentiable almost everywhere since the claims
	of Lemmas~\ref{Lemma:HorizDerivatives}, \ref{Lemma:DerivativesAlongCurves},
	\ref{Lemma:CoordinateDifferentiability} and \ref{Lemma:ApprDiffForLipschitz}
	hold almost everywhere in $\dom f$. This gives us an alternative proof
	of Theorem~\ref{Theorem:StepanoffLike}.
\end{remark}

Now we have all necessary tools to complete the proof of
Theorem~\ref{Theorem:Main}.
\begin{proof}[Proof of Theorem~\ref{Theorem:Main}]
\textsc{1st step.}
In the conditions of Theorem~\ref{Theorem:Main} the claims of
Lemmas~\ref{Lemma:HorizDerivatives},
\ref{Lemma:DerivativesAlongCurves} and \ref{Lemma:CoordinateDifferentiability} hold.
In particular $A_j = \dom \ap X_j f$ is a measurable set, $j=1,\dots,\dim H_1$,
$f$ is approximately differentiable along the curves $\Gamma_k(g;t)$ at $t=0$,
$k = 1,\dots,N$ almost everywhere in the set
$A = \bigcap\limits_{i=1}^{\dim H_1} A_j$
and relations \eqref{Eq:CurveDerivativeStructure} and \eqref{Eq:LDiffers} hold.

If \eqref{Eq:LDiffers} holds at the point $g \in A$
then, in view of structure of $L_g$ \eqref{Eq:LMapping},
estimate \eqref{Eq:CurveDerivativeEstimate} implies
\begin{align}
	& \ap \varlimsup_{v \to g} \frac {\widetilde d_{cc}(f(g), f(v))} {d_{cc}(g,v)} \notag \\
	& \,\, \leq \ap \varlimsup_{v \to g}
		\frac {\widetilde d_{cc}^{f(g)}(f(g), L_g(v))
		       + \widetilde d_{cc}^{f(g)}(L_g(v), f(v)) \bigr)
		       + o \bigl( \widetilde d_{cc}^{f(g)}(L_g(v), f(v)) \bigr) }
		{d_{cc}^g(g,v)}
	 \notag \\
	& \,\, \leq C \sup_{d_{cc}^g(g,v) \leq 1} \Bigl(
		\prod_{k=1}^N \tilde \delta^{f(g)}_{t_k} \ap d_{sub} (f \circ \Gamma_k)(g)
	\Bigr)
	< \infty. \label{TopCondition}
\end{align}
Hence, the left hand side of \eqref{TopCondition} is finite almost everywhere
in $A$. Applying Theorem~\ref{Theorem:SetFamilyLipschitz}, we obtain a countable
family of measurable sets covering $A$ up to the set of measure $0$ such that
the restriction of $f$ to each of them is a Lipschitz mapping.

Let $E$ be one set of this countable family and let
$L_g : \mathcal{G}^g \cap \mathcal{M} \to \mathcal{G}^{f(g)}$
be defined at almost all points of $E \subset A$.
To prove the theorem it remains to verify that $L_g$
is a homomorphism of the Lie groups.
In particular, we need to prove that given two points
$\hat u, \hat v \in \mathcal{G}^g$ we have
\begin{equation} \label{Eq:LIsHomomorphism}
	L_g(\hat u \cdot \hat v) = L_g(\hat u) \cdot L_g(\hat v).
\end{equation}

\textsc{2nd step.}
Let $g \in E$ be a density point where \eqref{Eq:LDiffers} holds
and suppose $B_{cc}(g,r_g) \subset \mathcal{G}^g$.
Then given $\hat v \in B_{cc}(g,r_g)$, $t \in [-r_g,r_g]$
there exists $v'_t = v'_t(g) \in E$, such that
$d_{cc}^g(\delta^g_t \hat v, v'_t) = o(t)$.
By Lemma~\ref{Lemma:ApprDiffForLipschitz}, we have
$$
	\ap \lim_{t \to 0} \frac {\widetilde d_{cc}^{f(g)}(f(v'_t), L_g(v'_t))} {t}
	= \lim_{t \to 0} \frac {\widetilde d_{cc}^{f(g)}(f(v'_t), L_g(v'_t))} {t}
	= 0.
$$
Then, using Property~\ref{PropertyL3} of the mapping $L_g$, we derive
\begin{multline*}
	\widetilde d_{cc}^{f(g)}(f(v'_t), L_g(\delta^g_t \hat v))
	\leq \widetilde d_{cc}^{f(g)}(f(v'_t), L_g(v'_t))
	     + \widetilde d_{cc}^{f(g)}(L_g(v'_t), L_g(\delta^g_t \hat v))\\
	=  o(d_{cc}^g(g,v'_t))+o(d_{cc}^g(g,\delta^g_t \hat v))
	= o(t) \quad \text{as~} t \to 0.
\end{multline*}
Next, consider two points $\hat u$, $\hat v \in B_{cc}(g,r_g / 2)$
and their product $\hat u \cdot \hat v$.
If  $\hat u = \widehat \Phi_g(s_1,\dots,s_N)$
and $\hat v = \widehat \Phi_g(r_1,\dots,r_N)$
then define by induction
\begin{align*}
	& u_1(t)(\cdot) = \Phi_1(t s_1)(\cdot); \\
	& u_k(t)(\cdot) = \Phi_k(t s_k)\circ u_{k-1}(t) (\cdot),
	  \quad k = 2, \dots, N; \\
	& v_1(t)(\cdot) = \Phi_1(t r_1)(\cdot); \\
	& v_k(t)(\cdot) = \Phi_k(t r_k) \circ v_{k-1}(t) (\cdot),
	  \quad k = 2, \dots, N.
\end{align*}
From the structure of functions $\Phi_k(\cdot)$ and from
Theorem~\ref{Th:PathsEstimate} it follows
\begin{align*}
	& d_{cc}^g(u_N(t)(g), \delta^g_t \hat u) = o(t), \\
	& d_{cc}^g(v_N(t)(g), \delta^g_t \hat v) = o(t), \\
	& d_{cc}^g(v_N(t) \circ u_N(t) (g), \delta^g_t (\hat u \cdot \hat v)) = o(t)
	  \quad \text{as~} t \to 0.
\end{align*}
As long as $g$ is a density point of $E$ we can find $w'_k(t)$,
$k = 1,\dots,2N$, such that
$d_{cc}^g \bigl( u_k(t)(g), w'_k(t) \bigr) = o(t)$
and
$d_{cc}^g \bigl( v_k(t) (u_N(t)(g)), w'_{N+k}(t) \bigr) = o(t)$ as $t \to 0$,
$k = 1,\dots,N$.
By the same arguments as above we conclude that
$$
	\widetilde d_{cc}^{f(g)}( f(w'_{2N}(t)), L_g(\delta^g_t [\hat u \cdot \hat v]) ) = o(t)
	\quad \text{as~} t \to 0.
$$
All we need is to verify that
\begin{equation} \label{EstimOfTruth}
	\widetilde d_{cc}^{f(g)}( f(w'_{2N}(t)), L_g(\delta^g_t \hat u) \cdot L_g(\delta^g_t \hat v) ) = o(t)
	\quad \text{as~} t \to 0.
\end{equation}

\textsc{3rd step.}
To prove \eqref{EstimOfTruth}
we assume $\mathcal{H}^\nu(E) < \infty$ and restrict the set $E$ applying
Egorov's and Luzin's theorems.

Recall that the mapping $x \mapsto \ap d_{sub}f \circ \Gamma_k(x)$
is defined in $E$, is measurable.
By Lemma~\ref{Lemma:ApprDiffForLipschitz} we get
\begin{equation} \label{Eq:PartialDD}
	\lim_{t \to 0} \widetilde d_{cc}^{f(x)}( f \circ \Phi_k(t)(x),
		\delta^{f(x)}_t \ap d_{sub} (f \circ \Gamma_k)(x) )
	= 0
\end{equation}
for every density point $x \in E$ as $t \to 0$,
$\Phi_k(t)(x) \in E$.

First, by Luzin's theorem there is
a closed set $E_1 \subset E$ such that
$\mathcal{H}^\nu(E \setminus E_1) < \varepsilon / 3$ and

\begin{itemize}
\item[(a)] 
all the mappings
$x \mapsto \ap d_{sub}f \circ \Gamma_k(x)$ are uniformly continuous
in $E_1$, $k = 1,\dots,N$.
\end{itemize}

Next, by Egorov's theorem there is a measurable set $E_2 \subset E_1$
such that $\mathcal{H}^\nu(E_1 \setminus E_2) < \varepsilon / 3$ and

\begin{itemize}
\item[(b)]
the
limit \eqref{Eq:PartialDD} converges uniformly on $E_2$, $k = 1,\dots,N$.
\end{itemize}

Now we consider a family of measurable functions
$$
 E_2 \owns x \to \psi_t(x)=\frac{\mathcal{H}^\nu(B_{cc}(x,t) \setminus E)}{\mathcal{H}^\nu(B_{cc}(x,t))}.
$$
We have that $\lim\limits_{t \to 0} \psi_t(x) = 0$ at almost all points of $x \in E_2$.
By Egorov's theorem there exists a measurable set $E_3 \subset E_2$ such that
$\mathcal{H}^\nu(E_2 \setminus E_3) < \varepsilon / 3$ and
the limit
\begin{itemize}
\item[(c)]
$\lim\limits_{t \to 0} \psi_t(x) = 0$ is uniform in $E_3$.
\end{itemize}

Property~(c) allows us to repeat the arguments
of the 2nd step with all $o(\cdot)$ uniform in $E_3$.
Therefore, if $x \in E_3$ we have
\begin{align*}
	& \widetilde d_{cc}^{f(x)} \bigl(
	    f(w'_1(t)(x)),
	    \delta^{f(x)}_{t \sigma_1} \ap d_{sub} (f \circ \Gamma_1)(x)
	  \bigr) = o(t), \\
	& \widetilde d_{cc}^{f(x)} \bigl(
	    f(w'_{k}(t)(w'_{k-1}(t))(x)),
	    \delta^{f(w'_{k-1}(t))}_{t \sigma_k} \ap d_{sub} (f \circ \Gamma_k)(w'_{k-1}(t)))
	  \bigr) = o(t), \\
	& \widetilde d_{cc}^{f(x)} \bigl(
	    f(w'_{N+1}(t)(w'_N(t))(x)),
	    \delta^{f(w'_N(t))}_{t \tau_1} \ap d_{sub} (f \circ \Gamma_1)(w'_N(t)))
	  \bigr) = o(t), \\
	& \widetilde d_{cc}^{f(x)} \! \bigl( 
	    f(w'_{\!N+k}\!(t)(w'_{\!N+k-1}\!(t))(x)),
	    \delta^{f(w'_{\!N+k-1}(t))}_{t \tau_k} \ap d_{sub} (f \!\circ\! \Gamma_k)(w'_{\!N+k-1}\!(t)))
	  \bigr) \!=\! o(t)
\end{align*}
as $t \to 0$, $k = 2,\dots,N$ and all $o(\cdot)$ are uniform with respect
to $x \in E_3$.
Here the coefficients $\sigma_k$ and $\tau_k$ are defined from \eqref{Eq:LMapping}
for the points $\hat u$ and $\hat v$ respectively.
Then, by properties (a) and (b) the relation
\begin{multline*}
	\widetilde d_{cc}^{f(x)} \Bigl(
		f(w'_{2N}(t)(x)),
		\prod_{k=1}^N \delta^{f(x)}_{t \sigma_k} \ap d_{sub} (f \circ \Gamma_k)(x)
		\cdot
		\prod_{k=1}^N \delta^{f(x)}_{t \tau_k} \ap d_{sub} (f \circ \Gamma_k)(x)
	\Bigr) \\
	= \widetilde d_{cc}^{f(x)} \bigl( f(w'_{2N}(t)(x)),
		\delta^{f(x)}_t L_x(\hat u) \cdot \delta^{f(x)}_t L_x(\hat v) \bigr)
	= o(t)
\end{multline*}
is uniform with respect to $x \in E_3$.
Finally,
\begin{align*}
	& t\, \widetilde d_{cc}^{f(x)} \bigl( L_x(\hat u \cdot \hat v), L_x(\hat u) \cdot L_x(\hat v) \bigr) \\
	& \quad = \widetilde d_{cc}^{f(x)} \bigl( \delta^{f(x)}_t L_x(\hat u \cdot \hat v),
	    \delta^{f(x)}_t L_x(\hat u) \cdot \delta^{f(x)}_t L_x(\hat v) \bigr) \\
	& \quad \leq \widetilde d_{cc}^{f(x)} \bigl( \delta^{f(x)}_t L_x(\hat u \cdot \hat v),
	    f(w'_{2N}(t)(x)) \bigr) \\
	& \qquad {} + \widetilde d_{cc}^{f(x)} \bigl( f(w'_{2N}(t)(x)),
	    \delta^{f(x)}_t L_x(\hat u) \cdot \delta^{f(x)}_t L_x(\hat v) \bigr) = o(t)
\end{align*}
and \eqref{TopCondition} is proved for $x \in E_3$. Since $\varepsilon$
is an arbitrary positive number, the Theorem is proved.
\end{proof}

\section{Application: an area formula}

Suppose that $x = \exp \Bigl( \sum\limits_{i=1}^N x_i X_i \Bigr)(g)$. Define
a quantity
\begin{multline}
	d_\rho (g,x) = \max \Bigl\{
	\Bigl( \sum_{j=1}^{\dim H_1} |x_j|^2 \Bigr)^{\frac 12},
	\Bigl( \sum_{j=\dim H_1 +1}^{\dim H_2} |x_j|^2 \Bigr)^{\frac 14},
	\dots, \\
	\Bigl( \sum_{j=\dim H_{M-1} + 1}^{N} |x_j|^2 \Bigr)^{\frac {1}{2 M}}
	\Bigr\}.
\end{multline}

It is easy to see that $d_\rho$ is locally equivalent to $d_\infty$.
Since we have already proved that $d_\infty$ and $d_{cc}$ are locally
equivalent, the following statement also holds.

\begin{proposition}
Let $g \in \mathcal{M}$. There is a compact neighborhood
$U(g) \subset \mathcal{M}$ such that
$$
	c_1 d_{cc}(u,v) \leq d_\rho(u,v) \leq c_2 d_{cc}(u,v)
$$
for all $u$, $v$ in $U(g)$, where constants $0 < c_1 \leq c_2 < \infty$
independent of $u$, $v \in U(g)$.
\end{proposition}

\begin{corollary}
Quantity $d_\rho$ is a quasimetric.
\end{corollary}

Denote an open ball in the quasimetric $d_\rho$ of raduis $r$ with center
in $x$ as $\Bx_\rho(x,r)$.
Define the (spherical) Hausdorff measure of a set $E$ with respect
to metric $d_\rho$ as
\begin{equation*}
	\mathcal{H}_\rho^k(E) = \lim_{\varepsilon \to 0+} \inf \Bigl\{
	\sum_i r^k_i : E \subset \bigcup_i \Bx_\rho(x_i, r_i), r_i < \varepsilon
	\Bigr\}.
\end{equation*}

Since Ball-Box theorem holds, Hausdorff measures constructed with respect
to $d_{cc}$ and with respect to $d_\rho$ are absolutely continuous one with
respect to another. We have
$$
	d \mathcal{H}_\rho^\nu(x) = \mathcal{D}_{\rho,cc}(x) \, d \mathcal{H}_{cc}^\nu(x),
	\quad x \in \mathcal{M},
$$
where $\mathcal{D}_{\rho,cc} : \mathcal{M} \to (0, \infty]$
is absolutely continuous and strictly positive. So, we could equally obtain
our results for $d_\rho$.

For Lipschitz mappings of Carnot--Carath\'eodory mappings the following
area formula holds.

\begin{theorem}[\cite{Bib:Karm-Area}] \label{Th:LipschitzArea}
Suppose $E \subset \mathcal{M}$ is a measurable set, and the mapping
$\varphi : E \to \widetilde{\mathcal{M}}$ is Lipschitz with respect to
sub-Riemannian quasimetrics $d_\rho$ and $\widetilde d_\rho$. Then the area
formula
\begin{equation} \label{Eq:LipschitzArea}
	\int\limits_E f(x) \mathcal{J}^{SR}(\varphi,x) d \mathcal{H}_\rho^\nu(x)
	= \int\limits_{\varphi(E)} \sum_{x :\: x \in \varphi^{-1}(y)} f(x) d\mathcal{H}_\rho^\nu(y)
\end{equation}
holds, where $f : F \to \mathbb{M}$ (here $\mathbb{M}$ is an arbitrary Banach
space) is such that function $f(x) \mathcal{J}^{SR}(\varphi,x)$ is
integrable, and 
\begin{equation} \label{Eq:JSR}
	\mathcal{J}^{SR}(\varphi,x) = \sqrt{\det (D\varphi(x)^* D\varphi(x))}
\end{equation}
is the sub-Riemannian Jacobian of $\varphi$ at $x$.
\end{theorem}

As an immediate corollary of \ref{Th:LipschitzArea} and \ref{Theorem:Main}
we obtain the following result.

\begin{theorem}
Suppose $E \subset \mathcal{M}$ is a measurable set, and the mapping
$\varphi : E \to \widetilde{\mathcal{M}}$ is approximately differentiable
almost everywhere. Then the area formula
$$
	\int\limits_E f(x) \ap \mathcal{J}^{SR}(\varphi,x) d \mathcal{H}_\rho^\nu(x)
	= \int\limits_{\widetilde{\mathcal{M}}}
	\sum_{x :\: x \in \varphi^{-1}(y) \setminus \Sigma} f(x) d\mathcal{H}_\rho^\nu(y)
$$
holds, where $f : F \to \mathbb{M}$ (here $\mathbb{M}$ is an arbitrary Banach
space) is such that function $f(x) \ap \mathcal{J}^{SR}(\varphi,x)$ is
integrable, $\mathcal{H}_\rho^\nu(\Sigma) = 0$ and
\begin{equation*}
	\ap \mathcal{J}^{SR}(\varphi,x) = \sqrt{\det ( \ap D\varphi(x)^* \ap D\varphi(x))}
\end{equation*}
is the approximate sub-Riemannian Jacobian of $\varphi$ at $x$.
\end{theorem}

\begin{proof}
By Theorem~\ref{Theorem:Main}, there is a sequence of disjoint sets
$\Sigma$, $E_1,E_2,\dots$ such that
$E = \Sigma \cup \bigcup\limits_{i=1}^\infty E_i$,
$\mathcal{H}_\rho^\nu(\Sigma) = 0$ and every restriction $\varphi|_{E_i}$
is a Lipschitz mapping. Then, by Theorem~\ref{Th:LipschitzArea}, we have
\begin{multline*}
	\int\limits_{E} f(x) \ap \mathcal{J}^{SR}(\varphi,x) d \mathcal{H}_\rho^\nu(x)
	= \sum_{i=1}^\infty \int\limits_{E_i} f(x) \mathcal{J}^{SR}(\varphi,x) d \mathcal{H}_\rho^\nu(x) \\
	= \sum_{i=1}^\infty \int\limits_{\widetilde{\mathcal{M}}}
	  \sum_{x:\: x \in \varphi^{-1}(y) \cap E_i} f(x) d\mathcal{H}_\rho^\nu(y)
	= \int\limits_{\widetilde{\mathcal{M}}}
	  \sum_{x:\: x \in \varphi^{-1}(y) \setminus \Sigma} f(x) d \mathcal{H}_\rho^\nu(y).
\end{multline*}
\end{proof}


\label{Page:DocumentEnd}


\begin{thebibliography}{NSW}
\addcontentsline{toc}{section}{References}

\bibitem[A]{Bib:Arnold}
Arnol'd V.~I.,
{\it Ordinary Differential Equations},
Springer-Verlag Berlin Heidelberg, 1992.

\bibitem[BLU]{Bib:Lanconelli}
Bonfiglioli A., Lanconelli E. and Uguzonni F.,
{\it Stratified Lie Groups and Potential Theory for their Sub-Laplasians},
Springer-Verlag Berlin Heidelberg, 2007.

\bibitem[Ch]{Bib:Chow}
Chow W.~L.,
{\it \"Uber systeme von linearen partiallen differentialgleichungen
erster ordnung},
Math. Ann. {\bf 117} (1939), 98--105.


\bibitem[F]{Bib:Federer}
Federer H.,
{\it Geometric measure theory},
series Die Grundlehren der mathematischen Wissenschaften,
{\bf Band 153}, New York: Springer-Verlag New York Inc.
(1969).

\bibitem[FS]{Bib:Folland-Stein}
Folland G.~B. and Stein E.~M.,
{\it Hardy spaces on homogeneous groups},
Princeton Univ. Press,
Princeton, 1982. (Math notes;~{\bf 28})

\bibitem[G]{Bib:Gromov}
Gromov M.,
{\it Carnot--Carath\'eodory spaces seen from within},
Sub-Riemannian Geometry, Prog. Math., vol. 144,
Birkh\"auser, Basel, 1996, 72--323.

\bibitem[Gr]{Bib:Greshnov}
Greshnov A.~V.,
{\it A proof of Gromov Theorem on homogeneous nilpotent approximation
for $C^1$-smooth vector fields},
Mathematicheskie Trudy, {\bf 15} (2012), no.\,2 (accepted)

\bibitem[H]{Bib:Hor}
H\"ormander L.,
{\it Hypoelliptic second order differential equations},
Acta Math. {\bf 119} (1967), 147--171.

\bibitem[K1]{Bib:Karm}
Karmanova M.,
{\it A new approach to investigation of Carnot--Carath\'eodory geometry}.
Geom. Func. Anal. {\bf 21}, no.~6 (2011), 1358--1374.

\bibitem[K2]{Bib:KarmG}
Karmanova M.,
{\it Convergence of scaled vector fields and local approximation theorem on
Carnot--Carath\'eodory spaces and applications}.
Dokl. AN {\bf 440} (2011), no.~6, 736--742.

\bibitem[K3]{Bib:Karm-Area}
Karmanova M.,
{\it The Area Formula for Lipschitz Mappings of Carnot--Carath\'eodory Spaces},
Izvestiya: Mathematics (submitted),
{\tt arXiv:1110.5483v1 [math.MG]}

\bibitem[KR]{Bib:KR}
Koranyi A. and Reimann H.~M.,
{\it Foundations for the theory of quasiconformal
mappings on the Heisenberg group}.
Adv. Math., {\bf111} (1995), 1--87.

\bibitem[KV]{Bib:Karm-Vod}
Karmanova M. and Vodopyanov S.,
{\it Geometry of Carnot--Carath\'eodory spaces, differentiability and coarea formula}.
Analysis and Mathematical Physics.
Birkh\"auser 2009, 284--387.

\bibitem[KV1]{Bib:Karm-Vod1}
Karmanova M. and Vodopyanov S.,
{\it On Local Approximation Theorem on Equiregular Carnot--Carath\'eodory spaces}.
Proceedings of "INDAM Meeting on Geometric Control and Sub-Riemannian Geometry", Cortona, May 2012.
Springer INdAM Series 2013 (to appear).

\bibitem[Li]{Bib:Lie}
Lie S. and Engel F.,
{\it Theorie der Transformationsgruppen}.
Bd 1-3, Lpz., 1888--93.

\bibitem[Ma]{Bib:Magnani}
Magnani V.,
{\it Differentiability and area formula on stratified Lie groups}.
Houston J. Math. {\bf 27} (2001), no.\,2, 297--323.

\bibitem[Me]{Bib:Metivier}
Metivier G.,
{\it Fonction spectrale et valeurs propres d'une classe d'op\'erateurs non
elliptiques}.
Commun. Partial Differential Equations {\bf 1} (1976), 467--519.

\bibitem[Mi]{Bib:Mitchell}
Mitchell J.,
{\it On Carnot--Carath\'eodory metrics}.
J. Differential Geometry {\bf 21} (1985), 35--45.

\bibitem[NSW]{Bib:NSW}
Nagel A., Stein E.~M. and Waigner S.,
{\it Balls and metrics defined by vector fields. I: Basic properties}.
Acta Math. (2) {\bf 155} (1985), 103--147.

\bibitem[P]{Bib:Pansu}
Pansu P.,
{\it Metriqu\'es de Carnot--Carath\'eodory et quasiisom\'etries des espaces
sym\'etriques de rang un}.
Ann. of Math. {\bf 119} (1989), 1--60.

\bibitem[Pa]{Bib:Pauls}
Pauls S.,
{\it A notion of rectifiability modeled on Carnot groups}.
Indiana Univ. Math. J. {\bf 53} (2004), 49--82.

\bibitem[Pos]{Bib:Postnikov}
Postnikov M.~M.,
{\it Lectures in Geometry. Semester V: Lie Groups and Lie Algebras.}
Moscow, ``Nauka'', 1982.

\bibitem[R]{Bib:Rademacher}
Rademacher H.,
{\it \"Uber partielle und totalle Differenzierbarkeit von Funktionen mehrerer
Variablen und \"uber die Transformation der Doppelintegrale}. I,
Math. Ann. {\bf 79} (1919), 340--359.

\bibitem[Ra]{Bib:Rashevsky}
Rashevsky P.~K.,
{\it Any two points of a totally nonholonomic space may be connected
by an admissable line},
Uch. Zap. Ped. Inst. im. Liebknechta, Ser. Phys. Math. {\bf 2} (1938), 83--94.

\bibitem[Re]{Bib:Reshetnyak}
Reshetnyak Yu.~G.,
{\it Generalized derivatives and differentiability almost everywhere}.
Math. USSR-Sb. {\bf 4} (1968), 293--302.

\bibitem[RS]{Bib:RS}
Rotchild L.~P. and Stein E.~S.,
{\it Hypoelliptic differential operators and nilpotent groups},
Acta Math. {\bf 137} (1976), 247--320.

\bibitem[S1]{Bib:StepDiff}
Stepanoff W.,
{\it \"Uber totale Differenzierbarkeit},
Math. Ann. {\bf 90} (1923), 318--320.

\bibitem[S2]{Bib:StepAppr}
Stepanoff W.,
{\it Sur les conditions de l'existence de la diff\'erentielle totale},
Matem. Sborn. Moscow {\bf 32} (1925), 511--527.

\bibitem[V1]{Bib:Vod96}
Vodop$'$yanov S.~K.,
{\it Monotone functions and quasiconformal mappings on Carnot groups}.
Siberian Math. J. {\bf 37} (1996), no.\,6, 1113--1136.

\bibitem[V2]{Bib:Vod99}
Vodop$'$yanov S.~K.,
{\it Mappings with bounded distortion and with finite distortion on Carnot groups.}
Siberian Math. J. {\bf 40} (1999), no.\,4, 644--677.

\bibitem[V3]{Bib:Vod-Groups}
Vodop$'$yanov S.~K.,
\textit{$\mathcal P$-differentiability on Carnot groups in different topologies and
related topics},
Proceedings on Analysis and Geometry (S. K. Vodopyanov, ed.)
Novosibirsk, Sobolev Institute Press, 2000, 603--670.

\bibitem[V4]{Bib:Vod07-1}
Vodop$'$yanov S. K.,
{\it Differentiability of mappings in the geometry of Carnot manifolds},
Siberian Math. J. {\bf 48} (2007), no.\,2, 197--213.

\bibitem[V5]{Bib:Vod-Spaces}
Vodopyanov S.~K.,
\textit{Geometry of Carnot--Carath\'eodory Spaces and Differentiability of Mappings},
Contemporary Mathematics {\bf 424} (2007), 247--302.

\bibitem[V6]{Bib:Vod07-2}
Vodopyanov S.~K.,
{\it Foundations of the theory of mappings with bounded distortion on
Carnot groups}.
The interaction of analysis and geometry, 303--344,
Contemp. Math., 424, Amer. Math. Soc., Providence, RI, 2007.

\bibitem[VG]{Bib:Vod-Greshnov}
Vodop$'$yanov S.~K. and Greshnov A.~V.,
{\it On the differentiability of mappings of Carnot--Carath\'eodory spaces},
Doklady Math. {\bf 67} (2003), no.\,2, 246--251.

\bibitem[VK1]{Bib:KV2}
Vodopyanov S.~K. and Karmanova  M.~B.,
{\it Local Geometry of Carnot Manifolds Under Minimal Smoothness}.
Doklady Math. {\bf 75} (2) (2007), 240--246.

\bibitem[VK2]{Bib:VK-Approx}
Vodop$'$yanov S.~K. and Karmanova M.~B.,
{\it Local approximation theorem on Carnot manifolds under minimal smoothness},
Dokl. AN {\bf 427} (2009), no.\,6, 731--736.

\bibitem[VU1]{Bib:Vod-Ukhlov}
Vodop$'$yanov S.~K. and Ukhlov A.~D.,
{\it Approximately differentiable tranformations and the change of
variables on nilpotent groups}.
Siberian Math. J. {\bf 37} (1996), no.\,1, 62--78.

\bibitem[VU2]{Bib:Vod-Ukhlov2}
Vodop$'$yanov S.~K. and Ukhlov A.~D.,
{\it Sobolev spaces and $(P,Q)$-quasiconformal mappings of Carnot groups}.
Siberian Math. J. {\bf 39} (1998), no.\,4, 665--682.

\bibitem[W]{Bib:Whitney}
Whitney H.,
\textit{On totally differentiable and smooth functions},
Pacific J. Math. {\bf 1} (1951), 143--159.

\end{thebibliography}
\end{document}